\documentclass[a4paper,10pt]{amsart}
\usepackage[top=0.8in, bottom=0.8in, left=1.25in, right=1.25in]{geometry}
\usepackage{fancybox}
\usepackage{amsmath}
\usepackage{amssymb}
\usepackage{amsthm}
\usepackage{hyperref}

\usepackage{mathrsfs}
\usepackage{epstopdf}
\usepackage[colorinlistoftodos]{todonotes}
\usepackage{exscale}
\usepackage{relsize}
\allowdisplaybreaks

\usepackage[all]{xy}
\usepackage{epstopdf}
\usepackage{array}
\usepackage{amscd}

\usepackage{amsfonts}
\usepackage{amssymb}
\usepackage{mathrsfs}
\usepackage{tikz}

\DeclareFontFamily{U}{rsfs}{%
\skewchar\font127}
\DeclareFontShape{U}{rsfs}{m}{n}{%
<-6>rsfs5<6-8.5>rsfs7<8.5->rsfs10}{}
\DeclareSymbolFont{rsfs}{U}{rsfs}{m}{n}
\DeclareSymbolFontAlphabet
{\mathrsfs}{rsfs}
\DeclareRobustCommand*\rsfs{%
\@fontswitch\relax\mathrsfs}

\theoremstyle{plain}
\newtheorem{thm}{Theorem}[section]
\newtheorem{prop}[thm]{Proposition}
\newtheorem{lem}[thm]{Lemma}

\newtheorem{prop-defi}[thm]{Proposition-Definition}
\newtheorem{thm-defi}[thm]{Theorem-Definition}
\newtheorem{lem-defi}[thm]{Lemma-Definition}
\newtheorem{cor-defi}[thm]{Corollary-Definition}

\newtheorem{conj}[thm]{Conjecture}

\theoremstyle{definition}
\newtheorem{defi}[thm]{Definition}
\newtheorem{exam}[thm]{Example}
\newtheorem{rmk}[thm]{Remark}

\newdimen\argwidth
\def\db[#1\db]{
 \setbox0=\hbox{$#1$}\argwidth=\wd0
 \setbox0=\hbox{$\left[\box0\right]$}
  \advance\argwidth by -\wd0
 \left[\kern.3\argwidth\box0 \kern.3\argwidth\right]}

\newcommand{\eE}{\mathcal{E}}

\newcommand{\hH}{\mathcal{H}}

\newcommand{\lL}{\mathcal{L}}

\newcommand{\oO}{\mathcal{O}}

\newcommand{\Hom}{\mathop{\rm Hom}\nolimits}

\newcommand{\dR}{\mathbf{R}}

\newcommand{\Hilb}{\mathop{\rm Hilb}\nolimits}

\newcommand{\Pic}{\mathop{\rm Pic}\nolimits}

\newcommand{\ch}{\mathop{\rm ch}\nolimits}

\newcommand{\Ext}{\mathop{\rm Ext}\nolimits}
\newcommand{\Spec}{\mathop{\rm Spec}\nolimits}

\newcommand{\Coh}{\mathop{\rm Coh}\nolimits}

\newcommand{\cneq}{\mathrel{\raise.095ex\hbox{:}\mkern-4.2mu=}}
\newcommand{\eqcn}{\mathrel{=\mkern-4.5mu\raise.095ex\hbox{:}}}

\newcommand{\DT}{\mathop{\rm DT}\nolimits}

\newcommand{\tr}{\mathop{\rm tr}\nolimits}

\newcommand{\RHom}{\mathop{\dR\mathrm{Hom}}\nolimits}

\makeatletter
 
  \@addtoreset{equation}{section}
\makeatother

\begin{document}
\title[Zero-dimensional $\DT_4$ invariants]{Zero-dimensional Donaldson-Thomas invariants \\ of Calabi-Yau 4-folds}
\author{Yalong Cao}
\address{Mathematical Institute, University of Oxford, Andrew Wiles Building, Radcliffe Observatory Quarter, Woodstock Road, Oxford, OX2 6GG }
\email{yalong.cao@maths.ox.ac.uk}
\author{Martijn Kool}
\address{Mathematical Institute, Utrecht University, P.O.~Box 80010 3508 TA Utrecht, the Netherlands}
\email{m.kool1@uu.nl}

\maketitle
\begin{abstract}
We study Hilbert schemes of points on a smooth projective Calabi-Yau 4-fold $X$. We define $\DT_4$ invariants by integrating the Euler class of a tautological vector bundle $L^{[n]}$ against the virtual class. We conjecture a formula for their generating series, which we prove in certain cases when $L$ corresponds to a smooth divisor on $X$. A parallel equivariant conjecture for toric Calabi-Yau 4-folds is proposed. This conjecture is proved for smooth toric divisors and verified for more general toric divisors in many examples. 

Combining the equivariant conjecture with a vertex calculation, we find explicit positive rational weights, which can be assigned to solid partitions. The weighted generating function of solid partitions is given by $\exp(M(q)-1)$, where $M(q)$ denotes the MacMahon function. 
\end{abstract}


\section{Introduction}

\subsection{Background}

Hilbert schemes on a smooth projective variety $X$ are moduli schemes which parametrize subschemes of $X$ with given Hilbert polynomial.
From the point of view of coherent sheaves, they can be regarded as moduli schemes of ideal sheaves of subschemes with fixed Chern character.
The simplest example is the Hilbert scheme $\Hilb^{n}(X)$ of $n$ points on $X$, whose ideal sheaves have Chern character $(1,0,\cdots,0,-n)$. There are lots of interesting studies on their geometry, topology and representation theory, most of which are concentrated on the cases $\dim_{\mathbb{C}}X\leqslant2$. The difficulty in extending these studies to higher dimensions comes from the fact that the Hilbert schemes are in general no longer smooth.

One surprising feature about $\dim_{\mathbb{C}}X=3$ is that, although $\Hilb^{n}(X)$ can be very singular with different irreducible components of various dimensions, it still carries a degree zero virtual class $[\Hilb^{n}(X)]^{\mathrm{vir}}$ \cite{MNOP}. The degree of this class is called a degree zero Donaldson-Thomas invariant of $X$ \cite{Thomas}.  An expression for the generating series of these invariants was conjectured and verified for local toric surfaces by Maulik-Nekrasov-Okounkov-Pandharipande \cite{MNOP} and confirmed in full generality by Levine-Pandharipande \cite{LP} and Li \cite{Li}. See also \cite{BF} for another proof in the Calabi-Yau case.

Our aim is to go one dimensional higher and restrict to the case of Calabi-Yau manifolds \cite{Yau}. By the work of Borisov-Joyce \cite{BJ} and 
Cao-Leung \cite{CL}, we have a virtual class construction for Gieseker moduli spaces of stable sheaves on smooth projective Calabi-Yau 4-folds, which is in particular applicable to 
$\Hilb^{n}(X)$. A difference from the case of 3-folds is that the virtual class is no longer of degree zero, so we need natural insertions
to define invariants.

\subsection{The compact case}

Let $X$ be a smooth projective Calabi-Yau 4-fold and let $\Hilb^n(X)$ denote the Hilbert scheme of $n$ points on $X$.
Assume the existence of an orientation $o(\mathcal{L})$ on the determinant line bundle $\mathcal{L}$ over $\Hilb^n(X)$. Then the results of~\cite{BJ, CL} provide a $\DT_4$ virtual class
\begin{align}\label{intro:DT4vir}
[\Hilb^n(X)]^{\rm{vir}}_{o(\mathcal{L})} \in H_{2n}(\Hilb^n(X), \mathbb{Z}). 
\end{align}
The virtual class (\ref{intro:DT4vir}) depends on the choice of orientation $o(\mathcal{L})$. On each connected component of $\Hilb^n(X)$, there are two choices of orientations, which affects the corresponding contribution to the class
(\ref{intro:DT4vir}) by a sign. We review facts about the $\DT_4$ virtual class in Section \ref{sec:review}. 

In order to define the invariants, we require 
insertions. Let $L$ be a line bundle on $X$ and denote by $L^{[n]}$ the tautological (rank $n$) vector bundle over 
$\Hilb^{n}(X)$ with fibre $H^0(L|_Z)$ over $Z \in \Hilb^n(X)$. 
Then it makes sense to define the following:
\begin{defi}
Let $X$ be a smooth projective Calabi-Yau 4-fold and let $L$ be a line bundle on $X$. Let $\mathcal{L}$ be the determinant line bundle of $\Hilb^n(X)$ with quadratic form $Q$ induced from Serre duality. Suppose $\mathcal{L}$ is given an orientation $o(\mathcal{L})$. 
We define 
\begin{equation}\DT_4(X,L,n\,;o(\mathcal{L})):=\int_{[\Hilb^n(X)]^{\mathrm{vir}}_{o(\mathcal{L})}}e(L^{[n]})\in \mathbb{Z},\quad \textrm{if}\textrm{ } n\geqslant1, \nonumber \end{equation}
where $e(-)$ denotes the Euler class. We also set $\DT_4(X,L,0\,;o(\mathcal{L})):=1$. 
\end{defi}
We make the following conjecture for the generating series of these invariants: 
\begin{conj}[Conjecture \ref{conj on zero dim dt4}] \label{conj intro}
Let $X$ be a smooth projective Calabi-Yau 4-fold and $L$ be a line bundle on $X$. There exist choices of orientation such that
\begin{equation}\sum_{n=0}^{\infty}\DT_4(X,L,n\,;o(\mathcal{L}))\,q^n=M(-q)^{\int_Xc_1(L) \cdot c_3(X)},  \nonumber \end{equation}
where 
\begin{equation}M(q)=\prod_{n=1}^{\infty} \frac{1}{(1-q^n)^n}\nonumber \end{equation} 
denotes the MacMahon function.
\end{conj}
We verify Conjecture \ref{conj intro} in some good cases based on the following geometric setting, where the line bundle $L=\mathcal{O}_X(D)$ is associated to an effective divisor $D\subseteq X$. 
\begin{prop}[Proposition \ref{section s}]\label{section s intro}
Let $X$ be a smooth quasi-projective variety, $D \subseteq X$ any effective divisor, and $L=\mathcal{O}_X(D)$. There exists a tautological section $\sigma$ and an isomorphism of schemes
\begin{align*}\xymatrix{
  & L^{[n]} \ar[d]_{\pi}  &  &  \\
  \sigma^{-1}(0)\cong \Hilb^{n}(D) \ar[r]^{\quad \quad \iota} & \Hilb^{n}(X).  \ar@/_1pc/[u]_{\sigma} &  &  }
\end{align*}
\end{prop}
For $n \leqslant 3$ and $D, X$ both smooth, the Hilbert schemes are smooth and we can explicitly compare deformation-obstruction theories on $X$ and $D$ (Proposition \ref{compare def-obs}). The latter gives rise to zero-dimensional 
$\DT_3$ invariants on $D$, which are known by the work of \cite{LP, Li}.
\begin{thm}[Theorem \ref{n less 4}] 
Let $X$ be a smooth projective Calabi-Yau 4-fold, $D \subseteq X$ a smooth divisor, and $L=\mathcal{O}_X(D)$. For each $n\leqslant3$, there exists a choice of orientation $o(\mathcal{L})$ such that 
\begin{equation} 
\int_{[\Hilb^{n}(X)]^{\mathrm{vir}}_{o(\mathcal{L})}}e(L^{[n]})=\int_{[\Hilb^{n}(D)]^{\mathrm{vir}}}1. \nonumber \end{equation}
In particular, Conjecture \ref{conj intro} is true in this setting.
\end{thm}
The proof for general $n$ will rely on Joyce's theory of D-manifolds or Kuranishi atlases. 
We hope to return to it in a future paper.

\subsection{The toric case}

When $X$ is a smooth quasi-projective toric Calabi-Yau 4-fold with action of $(\mathbb{C}^*)^4$, we can study an equivariant version of Conjecture \ref{conj intro}. Despite the non-compactness of $X$ and $\Hilb^n(X)$, we can still define an equivariant version of the $\DT_4$ virtual class on the torus fixed locus, which consists of a finite number of reduced points.

The definition involves the subtorus $T \subseteq (\mathbb{C}^*)^4$ preserving the Calabi-Yau volume form and hence Serre duality pairing. We note the following equality of fixed loci (Lemma \ref{Tfixlocus}, \ref{Tfixlocus scheme}) 
$$
\Hilb^n(X)^{T}=\Hilb^n(X)^{(\mathbb{C}^*)^4}.
$$ 
For any $Z\in \Hilb^n(X)^T$, 
we consider the equivariant Euler class
\begin{equation}e_T(\Ext^1(I_Z,I_Z))\in H^*(BT), \nonumber \end{equation}
and also the half Euler class
\begin{equation}e_T(\Ext^2(I_Z,I_Z),Q)\in H^*(BT), \nonumber \end{equation}
where $Q$ is the quadratic form induced from the Serre duality pairing on $\Ext^2(I_Z,I_Z)$.
We then have 
\begin{equation} \label{signeqn}
e_T(\Ext^2(I_Z,I_Z),Q) = \pm \sqrt{(-1)^{\frac{ext^2(I_Z,I_Z)}{2}}\,e_T\big(\Ext^2(I_Z,I_Z)\big)},
\end{equation}
where the class $(-)$ in $\sqrt{(-)}$ is a square and the sign depends on the choice of orientation.
\begin{defi}(Definition \ref{def of equ virtual class}) 
The $T$-equivariant virtual class of $\Hilb^{n}(X)$ is 
\begin{equation}[\Hilb^{n}(X)]_{T,o(\lL)}^{\mathrm{vir}}:=\sum_{Z\in\Hilb^{n}(X)^T}\frac{e_T\big(\Ext^{2}(I_Z,I_Z),Q\big)}{e_T\big(\Ext^{1}(I_Z,I_Z)\big)}, \nonumber \end{equation}
where $o(\lL)$ denotes a choice of sign in  (\ref{signeqn}) for each $Z\in\Hilb^{n}(X)^T$.

\end{defi}
By fixing a $T$-equivariant line bundle $L$ on $X$, we can consider the equivariant Euler class of its tautological bundle $e_T(L^{[n]})$ and define
\begin{equation}\DT_4(X,T,L,n\,;o(\lL)):=\sum_{Z\in\Hilb^{n}(X)^T}\frac{e_T\big(\Ext^{2}(I_Z,I_Z),Q\big)\cdot e_T(L^{[n]}|_Z)}{e_T\big(\Ext^{1}(I_Z,I_Z)\big)}. \nonumber \end{equation}
An equivariant version of Conjecture \ref{conj intro} can then be posed as follows:
\begin{conj}[Conjecture \ref{conj for toric}] \label{conj intro equi}
Let $X$ be a smooth quasi-projective toric Calabi-Yau 4-fold and $L$ be a $T$-equivariant line bundle on $X$. Then there exist choices of orientation $o(\mathcal{L})$ such that  
\begin{equation}\sum_{n=0}^{\infty}\DT_4(X,T,L,n\,;o(\mathcal{L}))\,q^n=M(-q)^{\mathlarger{\int}_Xc_1^{T}(L)\,\cdot\,c_3^{T}(X)}, \nonumber \end{equation}
where $\int_X$ denotes equivariant push-forward to a point.
\end{conj}
When $L=\oO_X(D)$ corresponds to a smooth toric divisor $D$, we can prove Conjecture \ref{conj intro equi}.
\begin{thm}[Theorem \ref{thm for equiv sm div}]\label{thm for equiv sm div intro}
Conjecture \ref{conj intro equi} is true for $L=\oO_X(D)$, where $D\subseteq X$ is a smooth $(\mathbb{C}^*)^4$-invariant divisor.
\end{thm}
Any smooth quasi-projective toric Calabi-Yau 4-fold $X$ can be covered by open $(\mathbb{C}^*)^4$-invariant subsets (equivariantly) isomorphic to $\mathbb{C}^4$. On each such subset, 
every $(\mathbb{C}^*)^4$-invariant zero-dimensional subscheme corresponds to a \textit{solid partition} $\pi = \{\pi_{ijk}\}_{i,j,k \geqslant 1}$, i.e.~a sequence of non-negative integers $\pi_{ijk} \in \mathbb{Z}_{\geqslant 0}$ satisfying
\begin{align*}
&\pi_{ijk} \geqslant \pi_{i+1,j,k}, \quad \pi_{ijk} \geqslant \pi_{i,j+1,k}, \quad \pi_{ijk}\geqslant \pi_{i,j,k+1} \quad \forall \textrm{ } 
i,j,k \geqslant 1, \end{align*}
\begin{equation}|\pi| := \sum_{i,j,k \geqslant 1} \pi_{ijk} < \infty, \nonumber \end{equation}
where $|\pi|$ is called the \emph{size} of $\pi$.

Using a vertex formalism as in MNOP \cite{MNOP}, we reduce Conjecture \ref{conj intro equi} to the case $X=\mathbb{C}^4$ (Proposition \ref{compare equi conj}). This leads us to assigning expressions $L_\pi(d_1,d_2,d_3,d_4)$ (coming from $e_T(L^{[n]})$) and $\mathsf{w}_\pi$ (coming from $e_T\big(\Ext^{2}(I_Z,I_Z),Q\big) / e_T\big(\Ext^{1}(I_Z,I_Z)\big)$) to any solid partition $\pi$. See Definition \ref{wpi}. In fact, the equivariant weight $\mathsf{w}_{\pi}$ is only defined up to sign, reflecting the different signs in \eqref{signeqn} for different choices of orientation. The case $X = \mathbb{C}^4$ then essentially corresponds to the following conjecture (which now includes a uniqueness assertion).
\begin{conj}[Conjectures \ref{affineconj} and \ref{affineconj unique}]\label{affineconj intro}
There exists a unique way of choosing the signs for the equivariant weights $\mathsf{w}_{\pi}$ such that
$$
\sum_{\pi} L_\pi(d_1,d_2,d_3,d_4) \, \mathsf{w}_\pi \, q^{|\pi|} = M(-q)^{\frac{(d_1 \lambda_1+d_2 \lambda_2+d_3 \lambda_3+d_4 \lambda_4)(-\lambda_1\lambda_2\lambda_3-\lambda_1\lambda_2\lambda_4-\lambda_1\lambda_3\lambda_4-\lambda_2\lambda_3\lambda_4)}{\lambda_1\lambda_2\lambda_3\lambda_4}}
$$
holds in $\frac{\mathbb{Q}(\lambda_1,\lambda_2,\lambda_3,\lambda_4)}{(\lambda_1+\lambda_2+\lambda_3+\lambda_4)}(d_1,d_2,d_3,d_4)[\![q]\!]$, where the sum is over all solid partitions and $M(q)$ denotes the MacMahon function.
\end{conj}
Besides Theorem \ref{thm for equiv sm div intro}, we verify Conjecture \ref{affineconj intro} in the following setting by using a Maple program, which calculates $\mathsf{w}_\pi$ for a given solid partition $\pi$.
\begin{thm}[Theorem \ref{verify affine conj}]
Conjecture \ref{affineconj intro} is true modulo $q^7$. 
\end{thm}

\subsection{Application to counting solid partitions}

By experimental study of many examples, we find that the specialization 
\begin{equation} \label{specialize}
L_\pi(0,0,0,-d) \, \mathsf{w}_\pi \Big|_{\lambda_1+\lambda_2+\lambda_3=0}
\end{equation}
is well-defined. 
We pose the following conjecture:
\begin{conj}[Conjecture \ref{specconj}]\label{specconj intro}
Let $\pi$ be a solid partition and let $\mathsf{w}_\pi$ be defined using the unique sign in Conjecture \ref{affineconj intro}. Then the following properties hold:
\begin{enumerate}
\item[(a)] $L_\pi(0,0,0,-d) \, \mathsf{w}_\pi \in \frac{\mathbb{Q}(\lambda_1,\lambda_2,\lambda_3,\lambda_4,d)}{(\lambda_1+\lambda_2+\lambda_3+\lambda_4)}$ has no pole at $\lambda_4 = -(\lambda_1+\lambda_2+\lambda_3)$.
\item[(b)] The specialization $L_\pi(0,0,0,-d) \, \mathsf{w}_\pi \Big|_{\lambda_1+\lambda_2+\lambda_3=0}$ is independent of $\lambda_1,\lambda_2,\lambda_3$.
\item[(c)] More precisely, there exists a rational number $\omega_\pi \in \mathbb{Q}_{>0}$ (independent of $d$) such that
\begin{equation} \label{Lwpieqintro}
L_\pi(0,0,0,-d) \, \mathsf{w}_\pi \Big|_{\lambda_1+\lambda_2+\lambda_3=0} = (-1)^{|\pi|} \, \omega_\pi \, \prod_{l=1}^{\pi_{111}} (d-(l-1)).
\end{equation}
In particular, for $d \in \mathbb Z_{>0}$, the LHS vanishes when $\pi_{111} > d$ and otherwise has the same sign as $(-1)^{\pi}$.
\end{enumerate}
\end{conj}

Geometrically, the specialization \eqref{specialize} corresponds to taking $X = \mathbb{C}^4$ and $D = \{x_4^d=0\} \subseteq \mathbb{C}^4$. Then $L = \oO(D) \cong \oO \otimes t_4^{-d}$. As we have seen in Proposition \ref{section s intro}, the canonical section of $L^{[n]}$ on $\Hilb^n(\mathbb{C}^4)$ cuts out the sublocus of zero-dimensional subschemes $Z$ contained in $D$. At the level of torus fixed points, we are therefore considering solid partitions $\pi$ of height $\pi_{111} \leqslant d$. This is the geometric motivation for the specialization \eqref{specialize}. 

We have the following evidence for this conjecture:
\begin{prop}[Proposition \ref{verify conj on poly dep}] \label{verify conj on poly dep intro} 
\hfill
\begin{itemize}
\item Conjecture \ref{specconj intro} is true for any solid partition $\pi$ of size $|\pi|\leqslant 6$.
\item Properties (a), (b), and the absolute value of equation \eqref{Lwpieqintro} hold for $d=1$ and any solid partition $\pi$ satisfying $\pi_{111}=1$ (in this case $|\omega_\pi| = 1$).     
\item Properties (a), (b), and the absolute value of equation \eqref{Lwpieqintro} hold for various individual solid partitions of size $\leqslant15$ listed in Appendix A.
\end{itemize}
\end{prop}
By combining Conjectures \ref{affineconj intro} and \ref{specconj intro}, we find a formula for enumerating $\omega_\pi$-weighted solid partitions $\pi$. 
\begin{thm}[Theorem \ref{solid partition counting}]\label{solid partition counting intro}
Assume Conjectures \ref{affineconj intro} and \ref{specconj intro} are true. Then
\begin{equation}\label{solid part count}
\sum_{\pi} \omega_\pi\,t^{\,\pi_{111}}\, q^{|\pi|} = e^{t(M(q)-1)},
\end{equation}
where the sum is over all solid partitions, $t$ is a formal parameter, and $M(q)$ denotes the MacMahon function. In particular, for $t=1$ 
\begin{equation}
\sum_{\pi} \omega_\pi\, q^{|\pi|} = e^{M(q)-1}. \nonumber 
\end{equation}
\end{thm}

This theorem inspired us to define an \emph{explicit} combinatorial weight $\omega_\pi^c \in \mathbb{Q}_{>0}$ associated to each solid partition $\pi$ (Definition \ref{combinatoric wpi}). Firstly, we \emph{prove} an unconditional version of Theorem \ref{solid partition counting intro} with $\omega_\pi$ replaced by $\omega_{\pi}^{c}$ (Theorem \ref{combinatorial wpi thm}). Secondly, we conjecture that $\omega_\pi = \omega_\pi^c$ and check this for the cases of Proposition \ref{verify conj on poly dep intro} (Conjecture \ref{compare wpi}, Proposition \ref{evidence compare wpi}). 

The definition of $\omega_\pi^c$ (Definition \ref{combinatoric wpi}) can naturally be extended to $d$-dimensional partitions for any $d\geqslant 0$, where $d=3$ corresponds to the case of solid partitions. 
The proof of Theorem \ref{combinatorial wpi thm} immediately gives
$$
\log \sum_{d\textrm{-partitions } \pi} \omega_\pi^c \, q^{|\pi|} =  \sum_{(d-1)\textrm{-partitions } \pi,\, |\pi|\geqslant1} q^{|\pi|}
$$
and we give a similar formula involving the formal parameter $t$ (Remark \ref{rmk on d dim combinatoric weight}). In a future work \cite{CK2}, we relate this formula to equivariant DT type invariants on $\mathbb{C}^{d+1}$. \\

There is a related work due to Nekrasov \cite{Nekrasov}, where he proposes a conjectural formula for a very general equivariant 
\textit{K-theoretical} partition function on Hilbert schemes of points on $\mathbb{C}^4$. Specializations of his partition function seem related to our Conjecture \ref{affineconj intro}. We briefly discuss a very special instance of his conjecture in Appendix B, where we point out relations to our choices of orientation. As opposed to \cite{Nekrasov}, our study of the $\mathbb{C}^4$ case emerges from first studying the compact case (Conjecture \ref{conj intro}) and subsequently studying the toric analogues (Conjectures \ref{conj intro equi} and \ref{affineconj intro}).

\subsection{Acknowledgement} 

This work was initiated during a visit of the first author to the Mathematical Institute of Utrecht University. He is grateful to the
institute for providing an excellent environment.
Y.~C.~is supported by The Royal Society Newton International Fellowship. 
We are very grateful to Professor Nikita Nekrasov for sending us his preprint and correspondence via e-mails. We also thank the anonymous referee for providing numerous suggestions to improve the exposition of the paper.

\section{The compact case}

Before stating our conjecture for Hilbert schemes of points on smooth projective Calabi-Yau 4-folds, we review the framework of $\mathrm{DT_ 4}$ invariants.

\subsection{Review of $\mathrm{DT_4}$ invariants}\label{sec:review}

Let $X$ be a smooth projective Calabi-Yau 4-fold, i.e.~a smooth projective 4-fold $X$ satisfying $K_X\cong \oO_X$ and $H^i(\oO_X) =0$ for $i=1,2,3$. Let $\omega$ be an ample divisor on $X$ and $v \in H^{\ast}(X, \mathbb{Q})$ a cohomology class.

The coarse moduli space $M_{\omega}(v)$
of $\omega$-Gieseker semistable sheaves
$E$ on $X$ with $\ch(E)=v$ exists as a projective scheme.
We always assume that
$M_{\omega}(v)$ is a fine moduli space, i.e.~any point $[E] \in M_{\omega}(v)$ is stable and
there is a universal family
\begin{align*}
\eE \in \Coh(X \times M_{\omega}(v)).
\end{align*}
In~\cite{BJ, CL}, under certain hypotheses,
the authors construct 
a $\mathrm{DT}_{4}$ virtual
class
\begin{align}\label{virtual}
[M_{\omega}(v)]^{\rm{vir}} \in H_{2-\chi(v, v)}(M_{\omega}(v), \mathbb{Z}), \end{align}
where $\chi(-,-)$ denotes the Euler pairing.
This class is not necessarily algebraic.

Roughly speaking, in order to construct such a class, one chooses at
every point $[E]\in M_{\omega}(v)$, a half-dimensional real subspace of the usual obstruction space
\begin{align*}\Ext_{+}^2(E, E)\subseteq \Ext^2(E, E)\end{align*}
on which the non-degenerate quadratic form $Q$ defined by Serre duality is real and positive definite. 
Then one glues local Kuranishi-type models of the form 
\begin{equation}\kappa_{+}=\pi_+\circ\kappa: \Ext^{1}(E,E)\to \Ext_{+}^{2}(E,E),  \nonumber \end{equation}
where $\kappa$ is a Kuranishi map of $M_{\omega}(v)$ at $E$ and $\pi_+$ is projection 
onto the first factor of 
\begin{equation} \label{real sub}
\Ext^{2}(E,E)=\Ext_{+}^{2}(E,E)\oplus\sqrt{-1}\cdot\Ext_{+}^{2}(E,E).  
\end{equation}

In \cite{CL}, local models are glued in three special cases: 
\begin{enumerate}
\item when $M_{\omega}(v)$ consists of locally free sheaves only; 
\item when $M_{\omega}(v)$ is smooth; 
\item when $M_{\omega}(v)$ is a shifted cotangent bundle of a derived smooth scheme. 
\end{enumerate}
In these cases, the corresponding virtual classes are constructed using either gauge theory or algebro-geometric perfect obstruction theory.

Assuming $M_\omega(v)$ can be given a $(-2)$-shifted symplectic structure, a general gluing construction was given by Borisov-Joyce \cite{BJ} based on Pantev-T\"{o}en-Vaqui\'{e}-Vezzosi's theory of shifted symplectic geometry \cite{PTVV} and Joyce's theory of derived $C^{\infty}$-geometry.
The corresponding virtual class is constructed using Joyce's
D-manifold theory (a machinery similar to Spivak's theory of derived smooth manifolds or Fukaya-Oh-Ohta-Ono's theory of Kuranishi space structures used in defining Lagrangian Floer theory).

To have a better understanding of what $\DT_4$ virtual classes look like, we briefly review the construction in situations (2) and (3) mentioned above:  
\begin{itemize}
\item 
When $M_{\omega}(v)$ is smooth, the obstruction sheaf
$$
\mathrm{Ob} := \mathcal{E}{\it{xt}}_{\pi_M}^2(\mathcal{E},\mathcal{E})
$$ 
is a vector bundle 
on $M_{\omega}(v)$ endowed with a non-degenerate quadratic form $Q$ induced by Serre duality, 
where $\pi_M : X \times M_\omega(v) \rightarrow M_\omega(v)$ denotes projection. 
A family version of \eqref{real sub} defines a real subbundle $\mathrm{Ob}^+ \subseteq \mathrm{Ob}$ on which $Q$ is positive definite and $\mathrm{Ob}\cong \mathrm{Ob}^+ \otimes _{\mathbb{R}} \mathbb{C}$ are isomorphic as vector bundles with quadratic forms \cite[Lem.~5]{EG}. Since $M_\omega(v)$ is smooth, the Zariski tangent space $\Ext^1(E,E)$ at any $[E] \in M_\omega(v)$ has the same dimension as $M_{\omega}(v)$, which implies that the local Kuranishi maps 
are zero. The $\DT_4$ virtual class is given by
\begin{equation} \label{virtforsmooth} [M_{\omega}(v)]^{\rm{vir}}=\mathrm{PD}\big(e(\mathrm{Ob},Q)\big), \end{equation}
where $e(\mathrm{Ob}, Q)$ denotes the half-Euler class of 
$(\mathrm{Ob},Q)$, i.e.~the Euler class of a real subbundle $\mathrm{Ob}^+$ and $\mathrm{PD}(-)$ denotes the 
Poincar\'e dual. Equality \eqref{virtforsmooth} holds up to a sign on each connected component. This sign is determined by the \emph{choice of orientation}, which we review below. Note that the half-Euler class satisfies 
\begin{equation}\label{half Euler}
e(\mathrm{Ob},Q)^{2}=(-1)^{\frac{\mathrm{rk}(\mathrm{Ob})}{2}}e(\mathrm{Ob}),  \textrm{ }\mathrm{if}\textrm{ } \mathrm{rk}(\mathrm{Ob})\textrm{ } \mathrm{is}\textrm{ } \mathrm{even}, \end{equation} 
\begin{equation}e(\mathrm{Ob},Q)=0, \textrm{ }\mathrm{if}\textrm{ } \mathrm{rk}(\mathrm{Ob})\textrm{ } \mathrm{is}\textrm{ } \mathrm{odd}. 
\nonumber \end{equation}
\item Suppose $M_{\omega}(v)$ is a shifted cotangent bundle of a derived smooth scheme. Roughly speaking, this means that at any closed point $[F]\in M_{\omega}(v)$, we have a Kuranishi map
\begin{equation}\kappa \colon
 \Ext^{1}(F,F)\to \Ext^{2}(F,F)=V_F\oplus V_F^{*},  \nonumber \end{equation}
which factors through a maximal isotropic subspace $V_F$ of $(\Ext^{2}(F,F),Q)$. Then the $\DT_4$ virtual class of $M_{\omega}(v)$ is, 
roughly speaking, the 
virtual class of the perfect obstruction theory formed by $\{V_F\}_{F\in M_{\omega}(v)}$. 
When $M_{\omega}(v)$ is furthermore smooth as a scheme, 
then it is
simply the Euler class of the vector bundle 
$\{V_F\}_{F\in M_{\omega}(v)}$ over $M_{\omega}(v)$. 
\end{itemize}

\subsection*{On orientations}

In order to construct the above virtual class (\ref{virtual}) with coefficients in $\mathbb{Z}$ (instead of $\mathbb{Z}_2$), we need an orientability result 
for $M_{\omega}(v)$, which is stated as follows.
Let  
\begin{equation}
 \lL:=\mathrm{det}(\dR \hH om_{\pi_M}(\eE, \eE))
 \in \Pic(M_{\omega}(v)) \quad  \nonumber
\end{equation}
be the determinant line bundle of $M_{\omega}(v)$, equipped with the non-degenerate symmetric pairing $Q$ induced by Serre duality.  An \textit{orientation} of 
$(\mathcal{L},Q)$ is a reduction of its structure group from $O(1,\mathbb{C})$ to $SO(1, \mathbb{C})=\{1\}$. In other words, we require a choice of square root of the isomorphism
\begin{equation}Q: \lL\otimes \lL \to \oO_{M_{\omega}(v)}  \nonumber \end{equation}
in order to construct the virtual class (\ref{virtual}). The virtual class \eqref{virtual} depends on the choice of orientation $o(\mathcal{L})$, so we write $[M_{\omega}(v)]^{\rm{vir}}_{o(\mathcal{L})}$ in order to stress this dependence.

An existence result of orientations is proved in \cite[Thm.~2.2]{CL2} for Calabi-Yau 4-folds $X$ such that $\mathrm{Hol}(X)=SU(4)$ and $H^{\rm{odd}}(X,\mathbb{Z})=0$. 
Notice that, if orientations exist, the different choices form a torsor for $H^{0}(M_{\omega}(v),\mathbb{Z}_2)$. 

In particular, when $M_\omega(v)$ is smooth, the choice of orientation on $\lL$ is equivalent to a choice of orientation of a real subbundle $\mathrm{Ob}^+ \subseteq \mathrm{Ob}$. By the homotopy equivalence $O(n,\mathbb{C})\sim O(n,\mathbb{R})$, the real subbundle is unique up to isomorphisms.

\subsection{Conjecture for $\mathrm{DT_{4}}$ invariants of $\Hilb^{n}(X)$}

Let $X$ be a smooth projective Calabi-Yau 4-fold. For a positive integer $n$, we consider the Hilbert scheme $\Hilb^{n}(X)$ of $n$ points on $X$.
It can be identified with the Gieseker moduli space of semistable sheaves with Chern character $(1,0,0,0,-n)\in H^{\mathrm{even}}(X)$,
which is a fine moduli space whose closed points parametrize ideal sheaves of points. 

Given a line bundle $L$ on $X$, we define its tautological bundle $L^{[n]}$ as follows \cite[Sect.~4.1]{Lehn1}
\begin{equation}L^{[n]}:=(\pi_M)_{*}\big(\mathcal{O}_{\mathcal{Z}_n}\otimes \pi_X^{*}L\big),   \nonumber \end{equation}
where $\mathcal{Z}_n\subseteq \Hilb^{n}(X)\times X$ denotes the universal subscheme and 
$\pi_M, \pi_X$ are projections from the product $\Hilb^{n}(X)\times X$ to each factor. Since $\pi_M$ is a flat finite morphism of degree $n$, $L^{[n]}$ is a rank $n$ vector bundle on $\Hilb^{n}(X)$ with fibre $H^0(L|_Z)$ over $Z \in \Hilb^n(X)$. Note that the (real) virtual dimension of $\Hilb^n(X)$ is $2n$ by \eqref{virtual}. Hence we define:
\begin{defi}\label{def of zero dim inv}
Let $X$ be a smooth projective Calabi-Yau 4-fold and $L$ a line bundle on $X$. Assume the determinant line bundle 
$\mathcal{L}$ of $\Hilb^n(X)$, with its non-degenerate quadratic form $Q$ induced from Serre duality, is given an orientation $o(\mathcal{L})$. 
We define 
\begin{equation}\DT_4(X,L,n\,;o(\mathcal{L})):=\int_{[\Hilb^n(X)]^{\mathrm{vir}}_{o(\mathcal{L})}}e(L^{[n]})\in \mathbb{Z},\quad \textrm{if}\textrm{ } n\geqslant1, \nonumber \end{equation}
and $\DT_4(X,L,0\,;o(\mathcal{L})):=1$.  
\end{defi}
We make the following conjecture for the corresponding generating series. 
\begin{conj}\label{conj on zero dim dt4}
Let $X$ be a smooth projective Calabi-Yau 4-fold and $L$ a line bundle on $X$. Then there exist choices of orientation such that 
\begin{equation}\sum_{n=0}^{\infty}\DT_4(X,L,n\,;o(\mathcal{L}))\,q^n=M(-q)^{\int_Xc_1(L) \cdot c_3(X)},  \nonumber \end{equation}
where $M(q)$ 
denotes the MacMahon function.
\end{conj}
\begin{rmk} 
When $L=\oO_X$, Conjecture \ref{conj on zero dim dt4} follows from the fact that $\oO_X^{[n]}$ has a nowhere vanishing section which sends $Z$ to $1_{Z}\in H^{0}(X,\oO_Z)$. Then $e(\oO_X^{[n]}) = c_1(\oO_X) = 0$. 
\end{rmk}

\subsection{Geometric motivation of the conjecture}

Let us consider the case when $L=\mathcal{O}_X(D)$ corresponds to an effective divisor $D\subseteq X$. 
The following proposition is similar to \cite[Sect.~A.2]{KT1}\footnote{We thank the anonymous referee for pointing out a proof which is significantly simpler than our original.}.
\begin{prop}\label{section s}
Let $D \subseteq X$ be any effective divisor on a smooth quasi-projective variety $X$ and let $L:=\oO_X(D)$. The rank $n$ vector bundle $L^{[n]}$ on $\Hilb^n(X)$ has a tautological section $\sigma$ whose zero locus is isomorphic to the Hilbert scheme 
$\Hilb^n(D)$ of $n$ points on $D$.
\end{prop}
\begin{proof}
Consider the universal subscheme
\begin{displaymath}
\xymatrix
{
& \mathcal{Z} \ar_p[ld] \ar^q[dr] \ar@^{(->}[r] &  \Hilb^{n}(X) \times X \\
\Hilb^n(X) & & X.
}
\end{displaymath}
Let $s : D \subseteq X$ be a section defining $D$. We claim that the tautological section $\sigma := p_* q^* s$ of $L^{[n]} = p_*q^*L$ has the required property, i.e.~we have an equality of schemes
$$
Z(\sigma) = \Hilb^n(D).
$$
In order to see this, it suffices to take any $T$-flat family 
\begin{displaymath}
\xymatrix
{
& \mathcal{Z}_T \ar_{p_T}[ld] \ar^{q_T}[dr] \ar@^{(->}[r] &  T \times X \\
T & & X
}
\end{displaymath}
with zero-dimensional length $n$ fibres and prove that
$$
\mathcal{Z}_T \subseteq T \times D \subseteq T \times X
$$
if and only if the corresponding morphism $f : T \rightarrow \Hilb^n(X)$ factors through $Z(\sigma)$. 

Now $f$ factors through $Z(\sigma)$ if and only if $f^*\sigma$ is the zero section of $f^* L^{[n]}$.  Note that $\mathcal{Z}_T = \mathcal{Z} \times_T \Hilb^n(X)$ and
$$
f^* \sigma = f^* p_* q^* s = p_{T*} q_{T}^{*} s.
$$
Therefore $f^* \sigma$ is the zero section if and only if $\mathcal{Z}_T \subseteq T \times D$ as required.
\end{proof}

Let $X$ be a smooth projective Calabi-Yau 4-fold with \emph{smooth} divisor $D \subseteq X$ and let $L = \oO_X(D)$. Ideally, if all moduli spaces are smooth of expected dimensions\,\footnote{Of course, this fantasy situation never occurs.}, i.e.~$\dim_{\mathbb{C}}\Hilb^n(D)=0$ and 
$\dim_{\mathbb{R}}\Hilb^n(X)=2n$, then the section $\sigma$ constructed in Proposition \ref{section s} is transverse to the zero section and we have 
\begin{equation}\int_{[\Hilb^n(X)]^{\mathrm{vir}}}e(L^{[n]})=\int_{[\Hilb^n(D)]^{\mathrm{vir}}}1, \nonumber \end{equation}
modulo a sign coming from the choice of orientation involved in defining the LHS. 
Then Conjecture \ref{conj on zero dim dt4} would follow from the generating 
series of zero-dimensional Donaldson-Thomas invariants of a smooth projective 3-fold $D$ \cite{LP, Li}
\begin{equation} \label{LPLi} \sum_{n=0}^{\infty}\Big(\int_{[\Hilb^n(D)]^{\mathrm{vir}}}1\Big)\,q^{n}=M(-q)^{\int_Dc_3(TD\otimes K_D)}  \nonumber \end{equation}
and equation \eqref{Chernmanip} below.

For later reference, we add the derivation of the equality
\begin{equation} \label{Chernmanip}
\int_D c_3(TD\otimes K_D) = \int_X c_1(L) \cdot c_3(TX).
\end{equation}
Indeed, from the short exact sequence
$$
0 \rightarrow TD \rightarrow TX|_D \rightarrow N_{D/X} \rightarrow 0
$$
and the fact that $N_{D/X} \cong \oO_D(D) \cong K_D$ ($X$ is Calabi-Yau), we obtain
$$
\int_D c(TD\otimes K_D) = \int_X c_1(L) \cdot \frac{c(TX \otimes L)}{c(L \otimes L)},
$$
where $c(-)$ denotes total Chern class. The degree 3 part of the fraction is easily calculated:
$$
c_3(TX) + c_1(TX) \cdot c_1(L)^2 = c_3(TX),
$$
where the last equality again uses the fact that $X$ is Calabi-Yau.

\subsection{Preparation on deformation and obstruction theories}

We need to compare deformation-obstruction theories of $\Hilb^n(X)$ and $\Hilb^n(D)$ in order to verify our conjecture. 
\begin{lem}\label{lem on D/X}
Let $X$ be a smooth projective variety and $i:D\hookrightarrow X$ be a smooth divisor. For any subscheme $Z\subseteq D$, we have a short exact sequence
\begin{equation}\label{SES of ideals} 0\to \mathcal{O}_{X}(-D)\to I_{Z,X} \to i_*I_{Z,D}\to 0 \end{equation}
of coherent sheaves on $X$, where $I_{Z,\star}$ is the ideal sheaf of $Z$ in $\star$ (\,$\star=X$ or $D$). 

Furthermore, if $Z$ is zero-dimensional, we have a long exact sequence
\begin{align}
\begin{split} \label{LES}
0&\to\Ext^{0}_X(i_*I_{Z,D},i_*\mathcal{O}_Z)\to\Ext^{0}_X(I_{Z,X},i_*\mathcal{O}_Z)\to H^{0}(\mathcal{O}_{Z}(D))\to \\
&\to\Ext^{1}_X(i_*I_{Z,D},i_*\mathcal{O}_Z)\to\Ext^{1}_X(I_{Z,X},i_*\mathcal{O}_Z)\to H^{1}(\mathcal{O}_{Z}(D))=0, 
\end{split}
\end{align}
and canonical isomorphisms 
\begin{equation}\Ext^{i}_X(i_*I_{Z,D},i_*\mathcal{O}_Z)\cong\Ext^{i}_X(I_{Z,X},i_*\mathcal{O}_Z)\quad \textrm{for}\textrm{ } i\geqslant2.
\nonumber \end{equation}
\end{lem}
\begin{proof}
Sequence (\ref{SES of ideals}) can be easily deduced from the short exact sequences 
\begin{align*}
0&\to \mathcal{O}_X(-D)\to  \mathcal{O}_X \to  \mathcal{O}_D\to 0, \\
0&\to I_{Z,X}\to  \mathcal{O}_X \to  i_*\mathcal{O}_Z\to 0, \\ 
0&\to I_{Z,D}\to  \mathcal{O}_D \to  \mathcal{O}_Z\to 0, 
\end{align*}
and diagram chasing.
Applying $\RHom_X(-,i_*\mathcal{O}_Z)$ to (\ref{SES of ideals}), we get a distinguished triangle 
\begin{equation}\RHom_X(i_*I_{Z,D},i_*\mathcal{O}_Z)\to\RHom_X(I_{Z,X},i_*\mathcal{O}_Z)\to\RHom_X(\mathcal{O}_{X}(-D),i_*\mathcal{O}_Z), \nonumber \end{equation}
whose cohomology gives the long exact sequence (\ref{LES}) and the desired canonical isomorphisms because $Z$ is zero-dimensional.
\end{proof}
\begin{lem}\label{def-obs on 3-fold}
Let $X$ be a smooth projective variety with $\dim_\mathbb{C}(X)\geqslant 3$ and let $L\to X$ be a line bundle on $X$. 
For any zero-dimensional subscheme $Z\subseteq X$,
we have canonical isomorphisms
\begin{align*}
&\Ext^{1}_X(I_{Z,X},I_{Z,X}\otimes L)_0\cong \Hom_X(I_{Z,X},\mathcal{O}_Z\otimes L)\cong\Ext^{1}_X(\mathcal{O}_Z,\mathcal{O}_Z\otimes L), \\
&\Ext^{2}_X(I_{Z,X},I_{Z,X}\otimes L)_0\cong \Ext^1_X(I_{Z,X},\mathcal{O}_Z\otimes L)\cong\Ext^{2}_X(\mathcal{O}_Z,\mathcal{O}_Z\otimes L).
\end{align*}
\end{lem}
\begin{proof}
We apply $\RHom_X(-,\mathcal{O}_Z\otimes L)$ to $0\to I_{Z,X}\to \mathcal{O}_X\to \mathcal{O}_Z\to 0$ and get the long exact sequence
\begin{align*}
0&\to \Hom_X(\mathcal{O}_Z,\mathcal{O}_Z\otimes L)\to \Hom_X(\mathcal{O}_X,\mathcal{O}_Z\otimes L)\to \Hom_X(I_{Z,X},\mathcal{O}_Z\otimes L)\to \\
&\to \Ext^1_X(\mathcal{O}_Z,\mathcal{O}_Z\otimes L)\to \Ext^1_X(\mathcal{O}_X,\mathcal{O}_Z\otimes L)\to \Ext^1_X(I_{Z,X},\mathcal{O}_Z\otimes L)\to \\
&\to \Ext^2_X(\mathcal{O}_Z,\mathcal{O}_Z\otimes L)\to \Ext^2_X(\mathcal{O}_X,\mathcal{O}_Z\otimes L)\to \Ext^2_X(I_{Z,X},\mathcal{O}_Z\otimes L)\to\cdots. 
\end{align*}
Since $\Hom_X(\mathcal{O}_Z,\mathcal{O}_Z\otimes L)\cong\Hom_X(\mathcal{O}_X,\mathcal{O}_Z\otimes L)$ and $H^{>0}(X,\mathcal{O}_Z\otimes L)=0$
for zero-dimensional subschemes $Z\subseteq X$, we obtain isomorphisms
\begin{equation}\label{iso1 on D}\Ext^{i}_X(I_{Z,X},\mathcal{O}_Z\otimes L)\cong \Ext^{i+1}_X(\mathcal{O}_Z,\mathcal{O}_Z\otimes L)\quad 
\textrm{ }\textrm{for}\textrm{ }i\geqslant0. \end{equation} 
In particular, for $\mathrm{dim}_{\mathbb{C}}(X) = 3$, we obtain 
\begin{equation} \label{dim3eqn}
\dim_{\mathbb{C}}\Ext^{2}_X(I_{Z,X},\mathcal{O}_Z\otimes L)=\dim_{\mathbb{C}}\Ext^{0}_X(\mathcal{O}_Z,\mathcal{O}_Z),
\end{equation}
where we used Serre duality $\Ext^{3}_X(\mathcal{O}_Z,\mathcal{O}_Z) \cong \Ext^{0}_X(\mathcal{O}_Z,\mathcal{O}_Z \otimes K_X)^*$. We will use this later.

Next we consider the following commutative diagram 
\begin{align}
\begin{split} \label{tracediag}
\xymatrix{    &  \dR \Gamma(L)[1] \ar@{=}[r]\ar[d] &
\dR \Gamma(L)[1] \ar[d] \\
\RHom_X(I_{Z,X}, \mathcal{O}_Z\otimes L) \ar[r] & \RHom_X(I_{Z,X}, I_{Z,X}\otimes L)[1] \ar[r] \ar[d] &
\RHom_X(I_{Z,X}, L)[1] \ar[d] \\
&  \RHom_X(I_{Z,X}, I_{Z,X}\otimes L)_0[1]  &  \RHom_X(\mathcal{O}_{Z}, L)[2], }
\end{split}
\end{align}
where the horizontal and vertical rows are distinguished triangles.
By taking cones, we obtain a distinguished triangle
\begin{equation*}
\RHom_X(I_{Z,X}, \mathcal{O}_Z\otimes L) \to \RHom_X(I_{Z,X}, I_{Z,X}\otimes L)_0[1] \to
\RHom_X(\mathcal{O}_Z, L)[2].
\end{equation*}
The long exact sequence of its cohomology gives an isomorphism
\begin{equation}\Ext^{1}_X(I_{Z,X},I_{Z,X}\otimes L)_0\cong \Hom_X(I_{Z,X},\mathcal{O}_Z\otimes L), \nonumber \end{equation}
where we used $\Ext^2_X(\oO_Z,L) \cong H^{n-2}(X,\mathcal{O}_Z\otimes K_X\otimes L^{-1}) = 0$ because $n=\mathrm{dim}_{\mathbb{C}}(X) \geqslant 3$ and similarly $\Ext^1_X(\oO_Z,L)=0$. Furthermore, we obtain an exact sequence
\begin{align}
\begin{split} \label{longexact}
0&\to\Ext^{1}_X(I_{Z,X}, \mathcal{O}_Z\otimes L)\to \Ext^{2}_X(I_{Z,X}, I_{Z,X}\otimes L)_0 \to \Ext^{3}_X(\mathcal{O}_Z,L)\to  \\
&\to \Ext^{2}_X(I_{Z,X}, \mathcal{O}_Z\otimes L)  \to\Ext^{3}_X(I_{Z,X}, I_{Z,X}\otimes L)_0 \to\cdots.  
\end{split}
\end{align}
When $\dim_\mathbb{C}(X)\geqslant 4$, $\Ext^{3}_X(\mathcal{O}_Z,L)\cong H^{n-3}(X,\mathcal{O}_Z\otimes K_X\otimes L^{-1})^*=0$ and we are done. 

When $\dim_\mathbb{C}(X)=3$, the trace map $\Ext^{0}_X(I_{Z,X}, I_{Z,X}\otimes L')\cong H^{0}(X,L')$ is an isomorphism for any line bundle $L'$ because $Z$ has codimension $>1$ (cf.~\cite[I, proof of Lem.~2]{MNOP}). Hence $\Ext^{3}_X(I_{Z,X}, I_{Z,X}\otimes L)_0=0$. Furthermore
\begin{align*} 
\dim_{\mathbb{C}} \Ext^{3}_X(\mathcal{O}_Z,L) &= \dim_{\mathbb{C}} H^0(X,\oO_Z) \\
&= \dim_{\mathbb{C}} \Ext^{0}_X(\mathcal{O}_Z,\mathcal{O}_Z) \\
&= \dim_{\mathbb{C}}\Ext^{2}_X(I_{Z,X}, \mathcal{O}_Z\otimes L),
\end{align*}
where the second equality uses $\Hom_X(\mathcal{O}_Z,\mathcal{O}_Z)\cong\Hom_X(\mathcal{O}_X,\mathcal{O}_Z)$ and the third equality uses \eqref{dim3eqn}. The exact sequence \eqref{longexact} yields the desired isomorphism
\begin{equation*}
\Ext^1_X(I_{Z,X},\mathcal{O}_Z\otimes L)\cong\Ext^{2}_X(I_{Z,X},I_{Z,X}\otimes L)_0. \qedhere \nonumber 
\end{equation*}
\end{proof} 

In the following lemma, we focus attention on $\Hilb^n(X)$, where $X$ is a smooth projective Calabi-Yau 4-fold and $n\leqslant3$. We recall that for any smooth projective variety $Y$ and $n\leqslant3$, the Hilbert scheme $\Hilb^n(Y)$ is smooth of dimension $\dim_{\mathbb{C}}(Y)\cdot n$ (e.g. \cite{Lehn2}). 
In fact, for a subscheme $Z$ of length $n \leqslant3$,  Lemma \ref{def-obs on 3-fold} implies
\begin{align*}
\dim_{\mathbb{C}}\Ext^1_X(I_{Z,X},I_{Z,X})_0=\dim_{\mathbb{C}}\Ext_X^0(I_{Z,X},\oO_Z) = 4n, \\
\dim_{\mathbb{C}}\Ext^1_D(I_{Z,D},I_{Z,D})_0=\dim_{\mathbb{C}}\Ext_D^0(I_{Z,D},\oO_Z) = 3n. 
\end{align*}
\begin{lem}\label{def-obs on 4-fold}
Let $X$ be a smooth projective Calabi-Yau 4-fold and let $i:D\hookrightarrow X$ be a smooth divisor. For any zero-dimensional subscheme $Z\subseteq D$
of length $\leqslant3$,
the exact sequence (\ref{LES}) in Lemma \ref{lem on D/X} breaks into an exact sequence and a canonical isomorphism
\begin{equation}0\to\Ext^{0}_X(i_*I_{Z,D},i_*\mathcal{O}_Z)\to\Ext^{0}_X(I_{Z,X},i_*\mathcal{O}_Z)\to H^{0}(\mathcal{O}_{Z}(D))\to 
0,  \nonumber \end{equation}
\begin{equation}\Ext^{1}_X(i_*I_{Z,D},i_*\mathcal{O}_Z)\cong \Ext^{1}_X(I_{Z,X},i_*\mathcal{O}_Z). \nonumber \end{equation}
Furthermore, using the isomorphism $\Ext^{1}_X(I_{Z,X},i_*\mathcal{O}_Z)\cong\Ext^{2}_X(I_{Z,X},I_{Z,X})_0$ of Lemma \ref{def-obs on 3-fold}, 
we obtain a canonical inclusion (constructed in the proof)
\begin{equation}\Ext^1_D(I_{Z,D},\mathcal{O}_Z)\hookrightarrow \Ext^{2}_X(I_{Z,X},I_{Z,X})_0 \nonumber \end{equation}
of a half-dimensional subspace which is isotropic with respect to the non-degenerate quadratic form $Q$ on $\Ext^{2}_X(I_{Z,X},I_{Z,X})_0$ defined by Serre duality.
\end{lem}
\begin{proof}
In the proof, we will use the following dimensions 
\begin{align} 
\begin{split} \label{dimensions}
\dim_{\mathbb{C}}\Ext^0_D(I_{Z,D},\mathcal{O}_Z)=3n, \quad \dim_{\mathbb{C}}\Ext^0_X(I_{Z,X},\mathcal{O}_Z)=4n, \\
\dim_{\mathbb{C}}\Ext^1_D(I_{Z,D},\mathcal{O}_Z)=3n. \quad \dim_{\mathbb{C}}\Ext^1_X(I_{Z,X},\mathcal{O}_Z)=6n. 
\end{split}
\end{align}
The first line follows from the fact that $\Hilb^n(X)$ and $\Hilb^n(D)$ are smooth for $n\leqslant 3$ and these are exactly the Zariski tangent spaces at $Z$. The second line can be seen in several ways. Firstly $\Ext^1_D(I_{Z,D},\mathcal{O}_Z) \cong \Ext^2_D(I_{Z,D},I_{Z,D})_0$ and $\Ext^1_X(I_{Z,X},\mathcal{O}_Z) \cong \Ext^2_X(I_{Z,X},I_{Z,X})_0$ by Lemma \ref{def-obs on 3-fold}, so it suffices to calculate the dimensions of the latter. By Hirzebruch-Riemann-Roch on $D$ we have
\begin{align*}
0 = \chi(\oO_D) - \chi(I_{Z,D}, I_{Z,D}) &= \mathrm{dim}_{\mathbb{C}}  \Ext^1_D(I_{Z,D},I_{Z,D})_0 - \mathrm{dim}_{\mathbb{C}}  \Ext^2_D(I_{Z,D},I_{Z,D})_0 \\
&= 3n - \mathrm{dim}_{\mathbb{C}}  \Ext^2_D(I_{Z,D},I_{Z,D})_0.
\end{align*}
By Hirzebruch-Riemann-Roch and Serre duality on $X$ we have
\begin{align*}
2n = \chi(\oO_X) - \chi(I_{Z,X}, I_{Z,X}) &= 2 \mathrm{dim}_{\mathbb{C}}  \Ext^1_X(I_{Z,X},I_{Z,X})_0 - \mathrm{dim}_{\mathbb{C}}  \Ext^2_X(I_{Z,X},I_{Z,X})_0 \\
&= 8n - \mathrm{dim}_{\mathbb{C}}  \Ext^2_X(I_{Z,X},I_{Z,X})_0.
\end{align*}
This establishes \eqref{dimensions}.

The spectral sequence 
\begin{equation}E^{p,q}_2=\Ext^p_D(I_{Z,D},\mathcal{O}_Z\otimes\wedge^{q}K_D)\Rightarrow \Ext^{p+q}_X(i_*I_{Z,D},i_*\mathcal{O}_Z)
\nonumber \end{equation}
gives an isomorphism 
\begin{equation}\label{iso0.5 on X}\Ext^0_D(I_{Z,D},\mathcal{O}_Z)\cong \Ext^{0}_X(i_*I_{Z,D},i_*\mathcal{O}_Z) \end{equation}
and an exact sequence
\begin{align}
\begin{split} \label{iso0.6 on X}
0&\to \Ext^1_D(I_{Z,D},\mathcal{O}_Z)\to \Ext^{1}_X(i_*I_{Z,D},i_*\mathcal{O}_Z) \to \Ext^0_D(I_{Z,D},\mathcal{O}_Z\otimes K_D)  \to \\
&\to \Ext^2_D(I_{Z,D},\mathcal{O}_Z)\to \Ext^{2}_X(i_*I_{Z,D},i_*\mathcal{O}_Z) \to \Ext^1_D(I_{Z,D},\mathcal{O}_Z\otimes K_D) \to 0,
\end{split}
\end{align}
where we use $\Ext^3_D(I_{Z,D},\mathcal{O}_Z)=0$ (see (\ref{iso1 on D})).

Combining \eqref{dimensions} and \eqref{iso0.5 on X}, we know the exact sequence (\ref{LES}) in Lemma \ref{lem on D/X} breaks into a short exact sequence and a canonical isomorphism
\begin{equation}0\to\Ext^{0}_X(i_*I_{Z,D},i_*\mathcal{O}_Z)\to\Ext^{0}_X(I_{Z,X},i_*\mathcal{O}_Z)\to H^{0}(\mathcal{O}_{Z}(D))\to 
0,  \nonumber \end{equation}
\begin{equation}\label{iso1 on X}\Ext^{1}_X(i_*I_{Z,D},i_*\mathcal{O}_Z)\cong \Ext^{1}_X(I_{Z,X},i_*\mathcal{O}_Z). \end{equation}
In particular, $\dim_{\mathbb{C}}\Ext^{1}_X(i_*I_{Z,D},i_*\mathcal{O}_Z)=6n$ by \eqref{dimensions}.  Therefore \eqref{dimensions} implies that the six term exact sequence \eqref{iso0.6 on X} splits into two short exact sequences and we obtain
\begin{equation}\label{iso2 on X} 0\to \Ext^1_D(I_{Z,D},\mathcal{O}_Z)\to \Ext^{1}_X(i_*I_{Z,D},i_*\mathcal{O}_Z) \to
\Ext^0_D(I_{Z,D},\mathcal{O}_Z\otimes K_D)  \to 0.  \end{equation}

Together (\ref{iso1 on X}) and (\ref{iso2 on X}) provide an inclusion 
\begin{equation}\Ext^1_D(I_{Z,D},\mathcal{O}_Z) \hookrightarrow \Ext^{1}_X(I_{Z,X},i_*\mathcal{O}_Z) \cong \Ext^{2}_X(I_{Z,X},I_{Z,X})_0, \nonumber \end{equation}
where the second isomorphism comes from Lemma \ref{def-obs on 3-fold}. We have obtained a canonical inclusion of a half-dimensional subspace (by \eqref{dimensions}).

Next, we check $\Ext^1_D(I_{Z,D},\mathcal{O}_Z)$ is an isotropic subspace of $\big(\Ext^2_X(I_{Z,X}, I_{Z,X})_0,Q\big)$ under this inclusion. 
Given $u \in \Ext^1_D(I_{Z,D}, \oO_Z)$, the corresponding
element in $\Ext^2_X(I_{Z,X}, I_{Z,X})_0$ is given by the composition
\begin{align*}
I_{Z,X} \stackrel{\alpha}{\to} i_{\ast}I_{Z,D} \stackrel{i_{\ast}u}{\to}
i_{\ast}\mathcal{O}_Z[1] \stackrel{\beta}{\to} I_{Z,X}[2],
\end{align*}
where $\alpha$ is the morphism constructed in (\ref{SES of ideals})
and $\beta$ is the obvious morphism.
Given another $u' \in \Ext^1_D(I_{Z,D}, \mathcal{O}_Z)$,
it is enough to show the vanishing of the composition
\begin{equation*}
I_{Z,X} \stackrel{\alpha}{\to} i_{\ast}I_{Z,D} \stackrel{i_{\ast}u}{\to}
i_{\ast}\mathcal{O}_Z[1] \stackrel{\beta}{\to} I_{Z,X}[2]
\stackrel{\alpha[2]}{\to} i_{\ast}I_{Z,D}[2] \stackrel{i_{\ast}u'[2]}{\to}
i_{\ast}\mathcal{O}_Z[3] \stackrel{\beta[2]}{\to} I_{Z,X}[4].
\end{equation*}
We claim
\begin{equation} \label{isopush}
\Ext^1_X(i_*\mathcal{O}_Z,i_*I_{Z,D})\cong\Ext^1_D(\mathcal{O}_Z,I_{Z,D}). 
\end{equation}
This implies that the composition
$i_{\ast}\mathcal{O}_Z[1] \stackrel{\beta}{\to} I_{Z,X}[2]
\stackrel{\alpha[2]}{\to} i_{\ast}I_{Z,D}[2]$ can be written as $i_{\ast}\gamma$,
for some $\gamma : \oO_Z \to I_{Z,D}[1]$. Therefore the composition
\begin{align*}
i_{\ast}
I_{Z,D} \stackrel{i_{\ast}u}{\to}
i_{\ast}\mathcal{O}_Z[1] \stackrel{\beta}{\to} I_{Z,X}[2]
\stackrel{\alpha[2]}{\to} i_{\ast}I_{Z,D}[2] \stackrel{i_{\ast}u'[2]}{\to}
i_{\ast}\mathcal{O}_Z[3]
\end{align*}
comes from $\Ext^3_D(I_{Z,D}, \mathcal{O}_Z)$ which is zero by (\ref{iso1 on D}).

We are left to show \eqref{isopush}. This follows at once from the spectral sequence
\begin{equation}E^{p,q}_2=\Ext^p_D(\mathcal{O}_Z, I_{Z,D} \otimes\wedge^{q}K_D)\Rightarrow \Ext^{p+q}_X(i_*\mathcal{O}_Z,i_*I_{Z,D}),
\nonumber \end{equation}
and 
\begin{equation*}
\Ext^0_D(\mathcal{O}_Z,I_{Z,D}\otimes \,K_D)\cong \Ext^3_D(I_{Z,D},\mathcal{O}_Z)^{*}=0, 
\end{equation*}
where the vanishing is by \eqref{iso1 on D}.
\end{proof}
Combining Lemma \ref{def-obs on 3-fold} and \ref{def-obs on 4-fold}, we deduce the following: 
\begin{prop}\label{compare def-obs}
Let $X$ be a smooth projective Calabi-Yau 4-fold and let $D\subseteq X$ be a smooth divisor. 
For any zero-dimensional subscheme $Z\subseteq D$ of length $\leqslant3$, 
we have short sequences
\begin{equation}0\to\Ext^{1}_D(I_{Z,D},I_{Z,D})_0\to \Ext^{1}_X(I_{Z,X},I_{Z,X})_0 \to H^{0}(\mathcal{O}_{Z}(D))\to 
0,  \nonumber \end{equation}
\begin{equation}0\to\Ext^{2}_D(I_{Z,D},I_{Z,D})_0\to \Ext^{2}_X(I_{Z,X},I_{Z,X})_0\to \Ext^{2}_D(I_{Z,D},I_{Z,D})_0^{*}\to 0, \nonumber \end{equation}
under which $\Ext^{2}_D(I_{Z,D},I_{Z,D})_0$ is a maximal isotropic subspace of $\Ext^{2}_X(I_{Z,X},I_{Z,X})_0$ with respect to the non-degenerate quadratic form $Q$ defined by Serre duality.
\end{prop}
\begin{proof}
By Lemma \ref{def-obs on 3-fold}, we have isomorphisms
\begin{equation}\Ext^{i+1}_Y(I_{Z,Y},I_{Z,Y}\otimes L)_0\cong \Ext^i_Y(I_{Z,Y},\mathcal{O}_Z\otimes L),\textrm{ }\textrm{for}\textrm{ }i=0,1 \textrm{ }\textrm{and}\textrm{ } Y=X, D.  \nonumber \end{equation}
Combining with Lemma \ref{def-obs on 4-fold}, we obtain the desired short exact sequences and an inclusion 
\begin{equation}\Ext^{2}_D(I_{Z,D},I_{Z,D})_0\hookrightarrow \Ext^{2}_X(I_{Z,X},I_{Z,X})_0 \nonumber \end{equation}
of a maximal isotropic subspace. 

This leads to the following commutative diagram
\begin{align*}\xymatrix{
0 \ar[r] & \Ext^{2}_D(I_{Z,D},I_{Z,D})_0 \ar[r] \ar[d]^{\exists\, t}& \Ext^{2}_X(I_{Z,X},I_{Z,X})_0 \ar[d]_{Q}^{\cong} \ar[r] & W \ar[r] & 0 \\
0 \ar[r] & W^* \ar[r] & \Ext^{2}_X(I_{Z,X},I_{Z,X})_0^* \ar[r] &
\Ext^{2}_D(I_{Z,D},I_{Z,D})_0^* \ar[r] & 0. }
\end{align*}
Note that the restriction $t$ of $Q$ is injective, hence also an isomorphism by dimension counting. 
Thus the quadratic form $Q$ gives an identification $W\cong \Ext^{2}_D(I_{Z,D},I_{Z,D})_0^{*}$.
\end{proof}
A positive real form $V_+$ on a complex even dimensional vector space $V$ with non-degenerate quadratic form $Q$ is a half-dimensional real subspace on which $Q$ is real and positive definite. When the obstruction space $\Ext^{2}_X(E,E)_0$ has a maximal isotropic subspace as in Proposition \ref{compare def-obs}, we can apply the following useful fact:
\begin{prop}\label{isotropic iso}
Let $V$ be an even dimensional complex vector space with a non-degenerate quadratic form $Q$ and let $V_{\mathrm{iso}}$ be a maximal isotropic subspace of $(V,Q)$. Then for any positive real form $V_+$ of $(V,Q)$, the composition 
\begin{equation}c:V_{\mathrm{iso}}\hookrightarrow V\to V_+ \nonumber \end{equation}
of the inclusion and projection is an isomorphism of the underlying real vector spaces.
\end{prop}
\begin{proof}
Since dimensions of $V_{\mathrm{iso}}$ and $V_+$ are the same, we only need to check that the map $c$ is injective. 
Take $v\in V_{\mathrm{iso}}$ which projects to zero in $V_+$. By 
\begin{equation}V=V_+\oplus \sqrt{-1}\cdot V_+, \nonumber \end{equation}
we know $v\in \sqrt{-1}\cdot V_+$. Then $Q(v,v)=0$, by the isotropic property, which implies that $v=0$ since $Q$ is negative definite on the 
subspace $\sqrt{-1}\cdot V_+$.
\end{proof}

\subsection{Verification in simple cases: $n\leqslant3$}

When the number $n$ of points satisfies $n\leqslant3$, the Hilbert schemes $\Hilb^{n}(X)$ and $\Hilb^{n}(D)$ are smooth of dimensions $4n$ and $3n$ respectively.  Our conjecture can then be verified by direct calculation.
 \begin{thm}\label{n less 4}
Let $X$ be a smooth projective Calabi-Yau 4-fold. Let $D$ be a smooth divisor on $X$ and set $L:=\mathcal{O}_X(D)$. 
For each $n\leqslant3$, there exists a choice of orientation $o(\mathcal{L})$ such that
\begin{equation}
\int_{[\Hilb^{n}(X)]^{\mathrm{vir}}_{o(\mathcal{L})}}e(L^{[n]})=\int_{[\Hilb^{n}(D)]^{\mathrm{vir}}}1. \nonumber \end{equation}
In particular, Conjecture \ref{conj on zero dim dt4} is true modulo $q^4$ for $L = \oO_X(D)$ and $D \subseteq X$ a smooth divisor.
\end{thm}
\begin{proof}
When $n\leqslant3$, the Hilbert schemes $\Hilb^{n}(X)$, $\Hilb^{n}(D)$ are smooth of dimensions $4n$ and $3n$ respectively.We have also seen that the obstruction sheaf $\mathrm{Ob}$ on $\Hilb^{n}(X)$ is locally free of rank $6n$ (\eqref{dimensions} and Lemma \ref{def-obs on 3-fold}).

Consider the quadric bundle $(\mathrm{Ob},Q)$, where $Q$ is the non-degenerate quadratic form defined by Serre duality. By \cite[Lem.~5]{EG}, we can choose a positive real form $\mathrm{Ob}^+$ 
of the quadric bundle $(\mathrm{Ob},Q)$, such that $\mathrm{Ob} \cong \mathrm{Ob}^+ \otimes_{\mathbb{R}} \mathbb{C}$ as quadric bundles. 
Then
\begin{equation}[\Hilb^{n}(X)]^{\mathrm{vir}}_{o(\mathcal{L})}=\mathrm{PD}\big(e(\mathrm{Ob}^+)\big)\in H_{2n}(\Hilb^{n}(X)) \nonumber \end{equation}
for an appropriate choice of orientation $o(\mathcal{L})$ in the definition of both sides. 
Therefore
\begin{eqnarray*}
\int_{[\Hilb^{n}(X)]^{\mathrm{vir}}_{o(\mathcal{L})}}e(L^{[n]})&=& \int_{[\Hilb^{n}(X)]}e(L^{[n]})\cdot e(\mathrm{Ob}^+) \\
&=& \int_{[\Hilb^{n}(D)]}e(\mathrm{Ob}^+)|_{\Hilb^{n}(D)},
\end{eqnarray*}
where the second equality follows from the fact that $\Hilb^{n}(D)\subseteq\Hilb^{n}(X)$ represents the Poincar\'e dual of the Euler class $e(L^{[n]})$ by Proposition \ref{section s}.

Next, we use the fact that the subspaces 
$$
\Ext^2_D(I_{Z,D},I_{Z,D})_0 \hookrightarrow \Ext^2_X(I_{Z,X},I_{Z,X})_0
$$
determine a maximal isotropic subbundle $V_{\mathrm{iso}} \subseteq \mathrm{Ob}|_{\Hilb^n(D)}$. Note that
$$
V_{\mathrm{iso}} \cong \mathrm{Ob}_{\Hilb^n(D)}
$$
is precisely the obstruction bundle of the perfect obstruction theory on $\Hilb^n(D)$ studied in \cite{MNOP}, whose fiber over $Z\in\Hilb^{n}(D)$ is $\Ext^2_D(I_{Z,D},I_{Z,D})_0$. By (a family version of) Proposition \ref{isotropic iso}, we have
$$
e(\mathrm{Ob}^+)|_{\Hilb^n(D)} = e(V_{\mathrm{iso}}) = e(\mathrm{Ob}_{\Hilb^n(D)}).
$$
Since $\Hilb^n(D)$ is smooth, we also have 
$$
[\Hilb^n(D)]^{\mathrm{vir}} = e(\mathrm{Ob}_{\Hilb^n(D)}) \cap [\Hilb^n(D)].
$$
Putting everything together, we deduce
\begin{align*}
\int_{[\Hilb^{n}(D)]}e(\mathrm{Ob}^+)|_{\Hilb^{n}(D)}=&\int_{[\Hilb^{n}(D)]}e(\mathrm{Ob}_{\Hilb^{n}(D)}) \\
=& \int_{[\Hilb^{n}(D)]^{\mathrm{vir}}}1. 
\end{align*}
The final statement of the proposition follows from \cite{LP, Li} and \eqref{Chernmanip}.
\end{proof}
For general $\Hilb^{n}(X)$, we need Joyce's theory of D-manifolds or Kuranishi atlases to prove a similar statement.
We hope to return to this in a future work.

\section{The toric case}

\subsection{Definition and conjecture}

Following \cite[Sect.~8]{CL}, we can similarly study zero-dimensional $\DT_4$ invariants of toric Calabi-Yau 4-folds (which are never compact). 

Let $X$ be a smooth quasi-projective toric Calabi-Yau 4-fold. By this we mean a smooth quasi-projective toric 4-fold $X$ satisfying $K_X \cong \oO_X$ and  $H^{>0}(\oO_X) = 0$. We also assume the fan contains cones of dimension 4. Such cones correspond to $(\mathbb{C}^*)^4$-invariant affine open subsets (equivariantly) isomorphic to $\mathbb{C}^4$.
Fix a Calabi-Yau volume form $\Omega$ on $X$ and denote by $T\subseteq (\mathbb{C}^*)^4$ the 3-dimensional subtorus which preserves $\Omega$.
Let $\bullet$ be $\Spec \mathbb{C}$ with trivial $(\mathbb{C}^*)^4$-action. We denote by $\mathbb{C} \otimes t_i$ the 1-dimensional $(\mathbb{C}^*)^4$-representation with weight $t_i$ and we write $\lambda_i \in H_{(\mathbb{C}^*)^4}^{\ast}(\bullet)$ for its $(\mathbb{C}^*)^4$-equivariant first Chern class. Then
\begin{align*}
H^*_{(\mathbb{C}^*)^4}(\bullet)=\mathbb{C}[\lambda_1, \lambda_2, \lambda_3,\lambda_4],  \end{align*} 
\begin{align*}
H^*_{T}(\bullet)=\mathbb{C}[\lambda_1, \lambda_2, \lambda_3,\lambda_4]/(\lambda_1+\lambda_2+\lambda_3+\lambda_4)
\cong \mathbb{C}[\lambda_1, \lambda_2,\lambda_3]. 
\end{align*} 
The $(\mathbb{C}^*)^4$-action and $T$-action both canonically lift to the Hilbert scheme $\Hilb^n(X)$ of $n$ points on $X$,
where $T$ preserves the Serre duality pairing (for compactly supported sheaves).

Let $L$ be a $T$-equivariant line bundle on $X$ and let $L^{[n]}$ be its tautological bundle with induced $T$-equivariant structure.
As in Definition \ref{def of zero dim inv}, we would like to evaluate the integral
\begin{equation}\int_{[\Hilb^n(X)]^{\mathrm{vir}}}e(L^{[n]}),\quad \textrm{for}\textrm{ } n\geqslant1. \nonumber \end{equation}
However, $\Hilb^n(X)$ is non-compact, so the usual virtual class is not well-defined. Nevertheless,
$\Hilb^n(X)$ is ``equivariantly compact'', i.e.~the $T$-fixed locus $\Hilb^n(X)^T$ is compact. In fact, it consists of finitely many points.
\begin{lem} \label{Tfixlocus}
At the level of closed points, we have
\begin{equation}\Hilb^n(X)^T=\Hilb^n(X)^{(\mathbb{C}^*)^4}, \nonumber \end{equation}
which consists of finitely many points.
\end{lem}
\begin{proof}
We cover $X$ by maximal $(\mathbb{C}^*)^4$-invariant open affine subsets $\{U_\alpha\}$ with centres at $(\mathbb{C}^*)^4$-fixed points. There exist coordinates $x_1,x_2,x_3,x_4$ on $U_\alpha\cong\mathbb{C}^4$, such that the action of $t\in(\mathbb{C}^*)^4$ on $U_\alpha$ is given by
\begin{equation}t \cdot x_i = t_i x_i, \quad \textrm{for all } i=1,2,3,4. \nonumber \end{equation}
Then the Calabi-Yau torus is given by
\begin{equation}T=\{t\in(\mathbb{C}^*)^4\textrm{ }|\textrm{ }t_1t_2t_3t_4=1\}  \nonumber \end{equation}
and we see that $U_\alpha$ is also $T$-invariant. Therefore it suffices to prove the lemma for $X = U_\alpha = \mathbb{C}^4$ with the standard torus action. 

The $(\mathbb{C}^*)^4$-invariant ideals in $\mathbb{C}[x_1,x_2,x_3,x_4]$ are precisely the monomial ideals. Clearly 
\begin{equation}\Hilb^n(X)^T \supseteq \Hilb^n(X)^{(\mathbb{C}^*)^4}. \nonumber \end{equation}
By considering the weight of $x^{n_1}_1 x^{n_2}_2 x^{n_3}_3 x^{n_4}_4$ under the action of $t\in\mathbb{C}^4$, it is easy to see that
any $T$-invariant ideal $I\subseteq \mathbb{C}[x_1,x_2,x_3,x_4]$ is of form
\begin{equation}I=\langle x^{n_{11}}_1 x^{n_{12}}_2 x^{n_{13}}_3 x^{n_{14}}_4 f_1(x_1x_2x_3x_4),\cdots, x^{n_{l1}}_1 x^{n_{l2}}_2 x^{n_{l3}}_3 x^{n_{l4}}_4 f_l(x_1x_2x_3x_4) \rangle, \nonumber \end{equation}
where $\{f_i(y)\}$ are polynomials of one variable with constant coefficient 1 and $n_{ij} \in \mathbb{Z}_{\geqslant 0}$. Suppose $I$ is $T$-invariant and corresponds to a zero-dimensional subscheme $Z$. Then the underlying reduced subscheme $Z_{\mathrm{red}}$ is a zero-dimensional $T$-invariant subset of $\mathbb{C}^4$, i.e.~$Z_{\mathrm{red}} = \{(0,0,0,0)\}$. Therefore $I$ is determined by its restriction to any Zariski open neighbourhood $U$ of $(0,0,0,0)$.  Take
$$
(0,0,0,0) \in U =  \{f_1(x_1x_2x_3x_4) \neq 0\} \cap \cdots \cap \{f_l(x_1x_2x_3x_4) \neq 0\}.
$$
The polynomials $f_i(x_1x_2x_3x_3)$ become invertible elements on $U$ and therefore
$$
I|_U = \langle x^{n_{11}}_1 x^{n_{12}}_2 x^{n_{13}}_3 x^{n_{14}}_4,\cdots, x^{n_{l1}}_1 x^{n_{l2}}_2 x^{n_{l3}}_3 x^{n_{l4}}_4 \rangle.
$$
We conclude that
$$
I = \langle x^{n_{11}}_1 x^{n_{12}}_2 x^{n_{13}}_3 x^{n_{14}}_4,\cdots, x^{n_{l1}}_1 x^{n_{l2}}_2 x^{n_{l3}}_3 x^{n_{l4}}_4 \rangle
$$
which shows $\Hilb^n(X)^T \subseteq \Hilb^n(X)^{(\mathbb{C}^*)^4}$ as sets.
\end{proof}
\begin{exam}
Consider $X = \mathbb{C}^4$ with standard torus action. Then 
$$
I = \langle x_1^3,x_2^3,x_3^3,x_4^3, x_1^2 x_2^2 x_3^2 x_4^2 + x_1x_2x_3x_4 \rangle
$$ defines a zero-dimensional $T$-invariant subscheme. According to the proof of Lemma \ref{Tfixlocus}, it is equal to  $\langle x_1^3,x_2^3,x_3^3,x_4^3, x_1 x_2 x_3 x_4 \rangle$. Indeed, we have
$$
x_1x_2x_3x_4 = [x_2^3 x_3^3 x_4^3] x_1^3 + [1-x_1x_2x_3x_4] (x_1^2 x_2^2 x_3^2 x_4^2 + x_1x_2x_3x_4).
$$ 
\end{exam}
${}$ \\

Let $U \cong \mathbb{C}^4$ be a maximal $(\mathbb{C}^*)^{4}$-invariant affine open subset of $X$. Choose coordinates $x_1, \ldots, x_4$ such that the action is given by
\begin{equation*}
t \cdot x_i = t_i x_i, \quad \textrm{for all } i =1,2,3,4.
\end{equation*}
The $T$-invariant (and therefore $(\mathbb{C}^*)^{4}$-invariant by Lemma \ref{Tfixlocus}) zero-dimensional subschemes of $U_\alpha$ can be labelled by solid partitions. 

\begin{defi}
A \textit{solid partition} $\pi = \{\pi_{ijk}\}_{i,j,k \geqslant 1}$ consists of a sequence of non-negative integers $\pi_{ijk} \in \mathbb{Z}_{\geqslant 0}$ satisfying
\begin{align*}
&\pi_{ijk} \geqslant \pi_{i+1,j,k}, \quad \pi_{ijk} \geqslant \pi_{i,j+1,k}, \quad \pi_{ijk}\geqslant \pi_{i,j,k+1} \quad \forall \textrm{ } 
i,j,k \geqslant 1, \end{align*}
such that
\begin{equation}|\pi| := \sum_{i,j,k \geqslant 1} \pi_{ijk} < \infty. \nonumber \end{equation}
Here $|\pi|$ is called the \emph{size} of $\pi$.
\end{defi}
Specifically, the zero-dimensional subscheme $Z_\pi$ corresponding to the solid partition $\pi = \{\pi_{ijk}\}_{i,j,k \geqslant 1}$ is defined by the monomial ideal
$$
I_{Z_\pi} := \langle x_1^{i-1} x_2^{j-1} x_3^{k-1} x_4^{\pi_{ijk}} \ | \ i,j,k \geqslant 1 \,\rangle
$$
and $|\pi|$ equals the length of $Z_\pi$. The $(\mathbb{C}^*)^4$-equivariant representation of $Z_{\pi}$ is given by 
\begin{equation} \label{Zpi}
Z_\pi = \sum_{i,j,k \geqslant 1} \sum_{l=1}^{\pi_{ijk}} t_1^{i-1} t_2^{j-1} t_3^{k-1} t_4^{l-1},
\end{equation}
where the sum is over all $i,j,k \geqslant1$ for which $\pi_{ijk}\geqslant 1$.  \\

In order to be able to apply Serre duality for $\Ext^*(I_Z,I_Z)$ on a non-compact toric Calabi-Yau 4-fold $X$, we will use the following lemma.
\begin{lem}\label{lem on toric identify deform-obs}
For any $Z\in \Hilb^n(X)^T$, we have isomorphisms of $T$-representations
\begin{equation}\Ext^i(I_Z,\oO_Z) \cong \Ext^{i+1}(I_Z,I_Z), \,\,\, i=0,1,2,
\nonumber \end{equation}
\begin{equation}\Ext^{i}(I_Z,I_Z)\cong \Ext^i(\oO_Z,\oO_Z), \,\,\, i=1,2,3, \quad \Ext^{4}(I_Z,I_Z)=0. \nonumber \end{equation}
\end{lem}
\begin{proof}
All morphisms in this proof are $T$-equivariant. By applying $\RHom(-,\oO_Z)$ to the short exact sequence, 
\begin{equation} \label{idealseq} 0\to I_Z\to \oO_X \to \oO_Z\to 0,   \end{equation}
we obtain isomorphisms
\begin{equation}\label{map i2}\Ext^i(I_Z,\oO_Z) \cong \Ext^{i+1}(\oO_Z,\oO_Z), \quad i\geqslant0, \end{equation}
where we use $H^{i\geqslant1}(\oO_X) = 0$. By applying $\RHom(I_Z,-)$ to \eqref{idealseq} 
we obtain an exact sequence  
\begin{equation}\label{Trepsseq}\cdots \to \Ext^i(I_Z,\oO_Z)\to\Ext^{i+1}(I_Z,I_Z)\to \Ext^{i+1}(I_Z,\oO_X)\to \cdots.  \end{equation}
By applying $\RHom(-,\oO_X)$ to \eqref{idealseq}, we find
\begin{align}
\begin{split} \label{map i1}
&\Hom(I_Z,\oO_X) = \Hom(\oO_X,\oO_X), \\
&\Ext^1(I_Z,\oO_X) = \Ext^2(I_Z,\oO_X) =\Ext^4(I_Z,\oO_X) = 0, \\
&\Ext^3(I_Z,\oO_X)\cong \Ext^4(\oO_Z,\oO_X), 
\end{split}
\end{align}
where we use $H^{i\geqslant1}(\oO_X) = 0$ and $\Ext^i(\oO_Z,\oO_X) = 0$ for $i \leqslant 3$ 
(by \cite[pp.~78]{Huy}, $\mathcal{E}xt^{i\leqslant3}(\oO_Z,\oO_X)=0$, so the vanishing follows from the local-to-global spectral sequence 
$H^p(X,\mathcal{E}xt^{q}(-,-))\Rightarrow \Ext^{p+q}(-,-)$
\cite[pp.~85, (3.16)]{Huy}). Combining with (\ref{Trepsseq}), we get the following isomorphisms and exact sequence 
\begin{align} 
\begin{split} \label{twists} 
&\Ext^0(I_Z,\oO_Z) \cong \Ext^{1}(I_Z,I_Z), \quad \Ext^1(I_Z,\oO_Z) \cong \Ext^{2}(I_Z,I_Z), \\
&0 \to \Ext^2(I_Z,\oO_Z)\to\Ext^{3}(I_Z,I_Z)\to \Ext^{3}(I_Z,\oO_X)\stackrel{\eta\,}{\to} \Ext^3(I_Z,\oO_Z)\to  \\ 
&\ \, \to\Ext^{4}(I_Z,I_Z)\to\Ext^{4}(I_Z,\oO_X)=0. 
\end{split}
\end{align}
For the first isomorphism of \eqref{twists}, we used $\Hom(I_Z,I_Z) \cong \Hom(I_Z,\oO_X)$. This follows from the fact that the isomorphism $H^0(\oO_X) \rightarrow \Hom(I_Z,\oO_X)$ of \eqref{map i1} factors through $H^0(\oO_X) \rightarrow \Hom(I_Z,I_Z)$ (see diagram \eqref{tracediag}\,\footnote{\label{equivtrace}Since $X$ is smooth and quasi-projective, any $(\mathbb{C}^*)^4$-equivariant coherent sheaf on $X$ has a finite 
$(\mathbb{C}^*)^4$-equivariant locally free resolution by \cite[Prop.~5.1.28]{CG}. Therefore we have $T$-equivariant trace maps as usual.}).

We claim that the map $\eta$ is an isomorphism. In fact, we have a commutative diagram
\begin{equation}
\xymatrix{\ar @{} [dr] |{} \Hom(I_Z,\oO_X[3]) \ar[d]^{i_1} \ar[r]^{\eta} & \Hom(I_Z,\oO_Z[3]) \ar[d]^{i_2} 
&  \\  \Hom(\oO_Z[-1],\oO_X[3]) \ar[r]^{\phi}
&\Hom(\oO_Z[-1],\oO_Z[3]), }
\nonumber \end{equation}
where $i_1$, $i_2$ are isomorphisms in (\ref{map i1}), (\ref{map i2}) respectively, and $\phi$ is the map in the exact sequence 
\begin{equation}\to \Ext^4(\oO_Z,I_Z)\to \Ext^4(\oO_Z,\oO_X)\stackrel{\phi\,}{\to} \Ext^4(\oO_Z,\oO_Z)\to 0,  \nonumber \end{equation}
obtained by applying $\RHom(\oO_Z,-)$ to (\ref{idealseq}). By Riemann-Roch and Serre duality, we have\,\footnote{\label{compactification}Although $X$ is non-compact, we can pass to a ``toric compactification'' $X \subset \overline{X}$, i.e.~a smooth projective toric 4-fold containing $X$ as a $(\mathbb{C}^*)^4$-invariant open subset. Since $Z \subset X$ has proper support, we get $(\mathbb{C}^*)^4$-equivariant isomorphisms $H^0(X,\mathcal{E}{\it{xt}}^4_X(\oO_Z,\oO_X)) \cong H^0(\overline{X},\mathcal{E}{\it{xt}}^4_{\overline{X}}(\oO_Z,\oO_{\overline{X}}))$ and $\Ext^*_X(\oO_Z,\oO_Z) \cong \Ext^*_{\overline{X}}(\oO_Z,\oO_Z)$.} 
\begin{align*}
\dim_{\mathbb{C}}\Ext^4(\oO_Z,\oO_X)&= \dim_{\mathbb{C}} H^0(X,\mathcal{E}{\it{xt}}^4(\oO_Z,\oO_X)) \\
&= \chi(\oO_Z,\oO_X)=n, \nonumber  \\
\dim_{\mathbb{C}}\Ext^4(\oO_Z,\oO_Z)&=\dim_{\mathbb{C}}\Ext^0(\oO_Z,\oO_Z)=n. \nonumber 
\end{align*} 
Therefore $\phi$ is an isomorphism and
so is $\eta$. We conclude that $\Ext^2(I_Z,\oO_Z) \cong \Ext^{3}(I_Z,I_Z)$ and $\Ext^4(I_Z,I_Z)=0$ by (\ref{twists}), which finished the proof.
\end{proof}
\begin{rmk}
Although a smooth quasi-projective toric Calabi-Yau 4-fold $X$ is non-compact, the sheaf $\oO_Z$ has proper support for any $Z \in \Hilb^n(X)^T$. Therefore, we can apply $T$-equivariant Serre duality to $\Ext^i(\oO_Z,\oO_Z)$\,\footnote{See footnote \ref{compactification}.}. Consequently, Lemma \ref{lem on toric identify deform-obs} allows us to apply $T$-equivariant Serre duality to $\Ext^{i}(I_Z,I_Z)$ for $i=1,2,3$. We will use this throughout the rest of this section.
\end{rmk}

Similarly to \cite[I, Lem.~6]{MNOP}, we have the following.
\begin{lem}\label{Tfixlocus scheme}
For any $Z\in \Hilb^n(X)^T$, we have an isomorphism of $T$-representations
\begin{equation}\Ext^0(I_Z,\oO_Z) \cong \Ext^{1}(I_Z,I_Z).  \nonumber \end{equation}
Moreover, $\Ext^0(I_Z,\oO_Z)^T=0$. In particular, the scheme $\Hilb^n(X)^T=\Hilb^n(X)^{(\mathbb{C}^*)^4}$ consists of finitely many \emph{reduced} points.
\end{lem}
\begin{proof}
The isomorphism $\Ext^0(I_Z,\oO_Z) \cong \Ext^{1}(I_Z,I_Z)$ was proved in Lemma \ref{lem on toric identify deform-obs}.

Next we show $\Ext^0(I_Z,\oO_Z)^T=0$. In fact it suffices to prove this when $X = \mathbb{C}^4$.
Then there exists a convenient basis for $\Ext^0(I_Z,\oO_Z)$ of $(\mathbb{C}^*)^4$-equivariant homomorphisms. This basis is described by combinatorial objects, which we call \emph{Haiman arrows}. See \cite{MS} (and also \cite{BK}) for details. 
These are arrows $\alpha$ in the character lattice $\mathbb{Z}^4$ such that:
\begin{itemize}
\item the tail $t(\alpha) \in \mathbb{Z}^4$ satisfies $(I_Z)_{t(\alpha)} \neq 0$, i.e.~it lies on a nonzero weight space of $I_Z$,
\item the head $h(\alpha) \in \mathbb{Z}^4$ satisfies $(\oO_Z)_{h(\alpha) + (n_1,n_2,n_3,n_4)} \neq 0$ for some $n_1,n_2,n_3,n_4 \geqslant0$.
\end{itemize}
Denote the standard basis of $\mathbb{Z}^4$ by
$$
e_1=(1,0,0,0), \quad e_2=(0,1,0,0), \quad e_3 = (0,0,1,0), \quad e_4=(0,0,0,1).
$$
Suppose $\alpha$ is a Haiman arrow such that the arrow defined by $t(\alpha) \pm e_i$, $h(\alpha) \pm e_i$, for some choice of $\pm$ and some basis vector $e_i$, is also a Haiman arrow. I.e.~the Haiman arrow $\alpha$ can be translated to another neighbouring Haiman arrow $\beta$. Then we call these Haiman arrows \emph{equivalent}. This induces an equivalence relation on the collection of all Haiman arrows. Next, we consider the collection $\mathcal{C}$ of equivalence classes $c$ of Haiman arrows such that all representatives $\alpha \in c$ satisfy $h(\alpha) \in (\oO_Z)_{h(\alpha)} \neq 0$. Then the elements of $\mathcal{C}$ are in 1-1 correspondence with a basis of $(\mathbb{C}^*)^4$-equivariant homomorphisms of $\Ext^0(I_Z,\oO_Z)$ as follows.
To each class $c \in \mathcal{C}$ we assign a module morphism $\phi_c : I_Z \rightarrow \oO_Z$, which is determined as follows. For each $\alpha \in c$ such that $t(\alpha)$ corresponds to a minimal homogeneous generator of $I_Z$, we define 
\begin{align*}
\phi_c(x^{t(\alpha)}) = x^{h(\alpha)}
\end{align*}
and all other minimal homogeneous generators are mapped to zero. Here we use multi-index notation $x^w := x_1^{w_1}x_2^{w_2}x_3^{w_3}x_4^{w_4}$. It is part of Haiman's theory that this is well-defined and defines a basis $\{\phi_c\}_{c \in \mathcal{C}}$ of $\Hom(I_Z,\oO_Z)$. Clearly the weight of $\phi_c$ equals
$$
h(\alpha) - t(\alpha),
$$
which is independent of the choice $\alpha \in c$. The statement we are after follows from the fact that any Haiman arrow $\beta$ with the property that $h(\beta) - t(\beta) = (n,n,n,n)$, for some $n$, is equivalent to a Haiman arrow $\gamma$ satisfying $(\oO_Z)_{h(\gamma)} = 0$, i.e.~$[\beta] \not\in \mathcal{C}$. We conclude $\Ext^0(I_Z,\oO_Z)^T=0$.
\end{proof}
\begin{exam} \label{maxsquared}
Suppose $I_Z := (x_1,x_2,x_3,x_4)^2$. Then $\mathcal{C}$ consists of 40 elements (implying that $\Ext^0(I_Z,\oO_Z)$ is 40-dimensional and $\Hilb^5(\mathbb{C}^4)$ is singular at $Z$). Explicitly, the basis $\phi_c$ described in the proof of the previous lemma consists of the following $40$ homomorphisms:
\begin{align*}
\phi_{ij} : x_i^2 \mapsto x_j, \quad \textrm{any other minimal homogeneous generator} \mapsto 0 \\
\phi_{abc} : x_a x_b \mapsto x_c, \quad \textrm{any other minimal homogeneous generator} \mapsto 0 
\end{align*} 
for all $i,j$ and $a,b,c$ with $a < b$. Observe that none of these homomorphisms has weight of the form $(n,n,n,n)$. 
Therefore $\Ext^0(I_Z,\oO_Z)^T=0$. 
\end{exam}

We continue with the definition of equivariant $\DT_4$ invariants.
For $Z\in \Hilb^n(X)^T$, one can form complex vector bundle
\begin{equation}
\begin{array}{lll}
      & \quad ET\times_{T}\Ext^{i}(I_Z,I_Z)
      \\  &  \quad \quad\quad \quad \downarrow \\   &   \quad  ET\times_{T}\{I_Z\}=BT
\end{array}\textrm{ }\textrm{for}\textrm{ } i=1,2, \nonumber\end{equation}
whose Euler class is the $T$-equivariant Euler class $e_T\big(\Ext^{i}(I_Z,I_Z)\big)$.

When $i=2$, the Serre duality pairing on $\Ext^{2}(\oO_Z,\oO_Z)$ defines a non-degenerate quadratic form $Q$ on $\Ext^{2}(I_Z,I_Z)$ (via Lemma \ref{lem on toric identify deform-obs}) and also on $ET\times_{T}\Ext^{2}(I_Z,I_Z)$ as $T$ preserves the Calabi-Yau volume form. We define 
\begin{equation}\label{half euler class equivariant}e_T\big(\Ext^{2}(I_Z,I_Z),Q\big)\in\mathbb{Z}[\lambda_1,\lambda_2,\lambda_3] \end{equation}
as the \textit{half Euler class} of $(ET\times_{T}\Ext^{2}(I_Z,I_Z),Q)$. By definition, this is the Euler class of its positive real form\,\footnote{I.e.~a half rank real subbundle on which $Q$ is real and positive definite.}, which exists 
because the classifying space $BT$ is simply connected. The half Euler class \eqref{half euler class equivariant} depends on 
a choice of orientation on a positive real form. \\

Following \cite[Sect.~8]{CL}, we can define the equivariant virtual class as follows:
\begin{defi}\label{def of equ virtual class} 
Let $X$ be a smooth quasi-projective toric Calabi-Yau 4-fold. Denote by $T\subseteq (\mathbb{C}^*)^4$ the three-dimensional subtorus which preserves the 
Calabi-Yau volume form.
The $T$-equivariant virtual class of $\Hilb^{n}(X)$ is 
\begin{equation}[\Hilb^{n}(X)]_{T,o(\lL)}^{\mathrm{vir}}:=\sum_{Z\in\Hilb^{n}(X)^T}\frac{e_T\big(\Ext^{2}(I_Z,I_Z),Q\big)}{e_T\big(\Ext^{1}(I_Z,I_Z)\big)}
\in\mathbb{Q}(\lambda_1,\lambda_2,\lambda_3), \nonumber \end{equation}
where $o(\lL)$ denotes a choice of orientation of a positive real form of $(ET\times_{T}\Ext^{2}(I_Z,I_Z),Q)$ for each $Z\in\Hilb^{n}(X)^T$. 
\end{defi}
Note that we have $\Ext^{i}(I_{Z},I_{Z})=\Ext^{i}(I_{Z},I_{Z})_0$ for $i=1,2$, because $H^{>0}(\oO_X) = 0$\,\footnote{See footnote \ref{equivtrace} on the existence of $T$-equivariant trace maps.}.
\begin{rmk}
For each $Z\in\Hilb^{n}(X)^T$, $o(\lL)$ is equivalent to the choice of sign in the square root (\ref{signeqn}).
If the number of fixed points $\Hilb^{n}(X)^T$ is $N$, the number of choices of $o(\lL)$ is $2^N$.
\end{rmk}

The $T$-equivariant version of Definition \ref{def of zero dim inv} is given as follows.
\begin{defi}
In the setup of Definition \ref{def of equ virtual class}, let $L$ be a $T$-equivariant line bundle on $X$ with corresponding tautological bundle
$L^{[n]}$ on $\Hilb^n(X)$. Then
\begin{equation}\DT_4(X,T,L,n\,;o(\lL)):=\sum_{Z\in\Hilb^{n}(X)^T}\frac{e_T\big(\Ext^{2}(I_Z,I_Z),Q\big)\cdot e_T(L^{[n]}|_Z)}{e_T\big(\Ext^{1}(I_Z,I_Z)\big)}\in\mathbb{Q}(\lambda_1,\lambda_2,\lambda_3),\, \textrm{if}\textrm{ } n\geqslant1, \nonumber \end{equation}
\begin{equation}\DT_4(X,T,L,0\,;o(\lL)):=1. \nonumber \end{equation}
\end{defi}
We recall the notion of equivariant push-forward for (not necessarily compact) manifolds with torus action (e.g.~toric Calabi-Yau 4-folds). In the compact case, this coincides with
the usual proper push-forward in the Atiyah-Bott localization formula.
\begin{defi}\label{equi push}
Let $X$ be a smooth manifold with $T\cong (\mathbb{C}^*)^k$-action such that the torus fixed locus $X^T$ consists of finite number of (necessarily reduced) points.
The equivariant push-forward of $\pi: X\to pt$ is 
\begin{equation}\int_X\textrm{ } : H^{*}_T(X) \to H^{*}_T(pt)_{\mathrm{loc}}, \quad\textrm{s.t.} \textrm{ }\int_X\alpha=\sum_{x\in X^T}\frac{\iota_x^*\alpha}{e_T(T_xX)}, \nonumber \end{equation}
where $H^{*}_T(pt)_{\mathrm{loc}}$ is the ring of fractions of $H^{*}_T(pt)$, which is isomorphic to $\mathbb{C}(\lambda_1,\cdots,\lambda_k)$ if we
identify $H^{*}_T(pt)\cong \mathbb{C}[\lambda_1,\cdots,\lambda_k]$, and $\iota_x: \{x\}\times_T ET\to X\times_T ET$ is the natural inclusion.
\end{defi} 
We propose the following $T$-equivariant version of Conjecture \ref{conj on zero dim dt4}.
\begin{conj}\label{conj for toric}
Let $X$ be a smooth quasi-projective toric Calabi-Yau 4-fold. Denote by $T\subseteq (\mathbb{C}^*)^4$ the three-dimensional subtorus which preserves the 
Calabi-Yau volume form. Let $L$ be a $T$-equivariant line bundle on $X$. Then there exist choices of orientation such that
\begin{equation}\sum_{n=0}^{\infty}\DT_4(X,T,L,n\,;o(\mathcal{L}))\,q^n=M(-q)^{\mathlarger{\int}_Xc_1^{T}(L)\,\cdot\,c_3^{T}(X)}, \nonumber \end{equation}
where $M(q)$ denotes the MacMahon function.
\end{conj}

\subsection{Proof for smooth toric divisors}

Let $L=\oO_X(D)$ for a $T$-invariant divisor $D\subseteq X$. 
Note that if $D$ is not $(\mathbb{C}^*)^4$-invariant, by the proof of Lemma \ref{Tfixlocus}, 
$D$ can locally be written as the sum of a $(\mathbb{C}^*)^4$-invariant divisor and a $T$-invariant divisor which is not $(\mathbb{C}^*)^4$-invariant. Hence, locally near each fixed point, $L$ is $T$-equivariantly isomorphic to a $(\mathbb{C}^*)^4$-equivariant line bundle. Therefore it suffices to consider Conjecture \ref{conj for toric} for $(\mathbb{C}^*)^4$-equivariant divisors only. 

We prove Conjecture \ref{conj for toric} when $D \subseteq X$ is a \emph{smooth} $(\mathbb{C}^*)^4$-equivariant divisor.
\begin{thm}\label{thm for equiv sm div}
Let $X$ be a smooth quasi-projective toric Calabi-Yau 4-fold. Denote by $T\subseteq (\mathbb{C}^*)^4$ the three-dimensional subtorus which preserves the Calabi-Yau volume form. Let $L = \oO_X(D)$, where $D \subseteq X$ is a smooth $(\mathbb{C}^*)^4$-invariant divisor.
Then Conjecture \ref{conj for toric} is true.
\end{thm}
\begin{proof}
For $Z\in\Hilb^n(X)^T$ such that $Z\not\subseteq D$, i.e.~$Z$ does not lie scheme theoretically in $D$, we claim that 
\begin{equation}\label{vani} e_T(L^{[n]}|_Z)=0. \end{equation}
Let $U \cong \mathbb{C}^4$ be any $(\mathbb{C}^*)^4$-invariant affine open subset of $X$. As $D$ is smooth and $(\mathbb{C}^*)^4$-invariant, we can choose coordinates $x_1,x_2,x_3,x_4$ on $U$ such that the action is given by
\begin{equation}t \cdot x_i = t_i x_i, \quad \textrm{for all } i=1,2,3,4, \nonumber \end{equation}
and $D\cap U$ is defined by $x_4=0$. Equation \eqref{vani} then follows from Lemma \ref{vanishing} below.

Now we only need to calculate 
\begin{equation}\label{equ sum}\sum_{Z\in\Hilb^n(X)^T,\, Z\subseteq D}\frac{e_T(\Ext^{2}_X(I_{Z,X},I_{Z,X}),Q)\cdot e_T(L^{[n]}|_Z)}{e_T(\Ext^{1}_X(I_{Z,X},I_{Z,X}))}. 
\end{equation}
For $Z\in\Hilb^n(X)^T$ and $Z\subseteq D\subseteq X$, Lemma \ref{lem on toric identify deform-obs} gives $T$-equivariant isomorphisms 
\begin{align*}\Ext^{i}_X(I_{Z,X},I_{Z,X}) &\cong \Ext^{i}_X(\oO_Z,\oO_Z), \textrm{ }\textrm{for}\textrm{ }i=1,2,3, \nonumber \\
\Ext^{i}_D(I_{Z,D},I_{Z,D}) &\cong \Ext^{i}_D(\oO_Z,\oO_Z), \textrm{ }\textrm{for}\textrm{ }i=1,2, \nonumber \end{align*}
where the isomorphisms on $D$ can be deduced similarly as for $X$.

From the $T$-equivariant distinguished triangle (e.g. \cite[Cor.~11.4, pp.~248--249]{Huy})
\begin{equation}\RHom_D(\oO_Z,\oO_Z)\to \RHom_X(\oO_Z,\oO_Z)\to \RHom_D(\oO_Z,\oO_Z\otimes K_D)[-1], \nonumber \end{equation}
we obtain a $T$-equivariant exact sequence
\begin{align*}
0&\to \Ext^1_D(\oO_Z,\oO_Z)\to \Ext^1_X(\oO_Z,\oO_Z)\to \Hom_D(\oO_Z,\oO_Z\otimes K_D)\to \nonumber \\
&\to \Ext^2_D(\oO_Z,\oO_Z)\to \Ext^2_X(\oO_Z,\oO_Z)\to \Ext^1_D(\oO_Z,\oO_Z\otimes K_D)\to \nonumber \\
&\to \Ext^3_D(\oO_Z,\oO_Z)\to \Ext^3_X(\oO_Z,\oO_Z)\to \Ext^2_D(\oO_Z,\oO_Z\otimes K_D)\to 0. \nonumber 
\end{align*}
By $T$-equivariant Serre duality, this gives 
\begin{equation}\Ext^{1}_X-\Ext^{2}_X+\Ext^{3}_X=\Ext^{1}_D+(\Ext^{1}_D)^*-(\Ext^{2}_D+(\Ext^{2}_D)^*) 
\nonumber \end{equation}
\begin{equation}+H^0(D,\oO_Z\otimes K_D)+H^0(D,\oO_Z\otimes K_D)^* \in K_T(\bullet) 
\nonumber \end{equation}
in the $T$-equivariant $K$-theory of a point, where we abbreviate $\Ext^{i}_A:=\Ext^{i}_A(\oO_Z,\oO_Z)$. 
For the corresponding Euler classes, we deduce
\begin{equation}\frac{e_T(\Ext^{1}_X)\cdot e_T(\Ext^{3}_X)}{e_T(\Ext^{2}_X)}=(-1)^{n}\cdot\Bigg(\frac{e_T(\Ext^{1}_D)\cdot e_T(H^{0}(D,\mathcal{O}_Z\otimes K_D))}{e_T(\Ext^{2}_D)}\Bigg)^2
. \nonumber \end{equation}
Therefore we have 
\begin{eqnarray*}\frac{e_T(\Ext^{2}_X(I_{Z,X},I_{Z,X}),Q)\cdot e_T(L^{[n]}|_Z)}{e_T(\Ext^{1}_X(I_{Z,X},I_{Z,X}))}  
&=& \frac{e_T(\Ext^{2}_X(I_{Z,X},I_{Z,X}),Q)\cdot e_T(H^{0}(X,\mathcal{O}_Z\otimes\mathcal{O}_{X}(D)))}{e_T(\Ext^{1}_X(I_{Z,X},I_{Z,X}))} \\
&=& \frac{e_T(\Ext^{2}_D(I_{Z,D},I_{Z,D}))}{e_T(\Ext^{1}_D(I_{Z,D},I_{Z,D}))},
\end{eqnarray*}
where we used \eqref{signeqn} and $L|_D = K_D$ ($X$ is Calabi-Yau). Moreover, the second equality is up to sign corresponding to the choice of orientation in defining the half Euler class.

Being a toric prime divisor, $D \subseteq X$ is itself a smooth toric 3-fold \cite[Sect.~3.1]{Ful}. As above, on any $(\mathbb{C}^*)^4$-invariant open $U \cong \mathbb{C}^4$ we can choose coordinates such that $t \cdot x_i = t_i x_i$, for all $i=1,2,3,4$, and $D \cap U = \{x_4=0\}$. In these coordinates, the torus of $D$ is obtained from $T = \{t_1t_2t_3t_4=1\}$ by setting $t_4=1$, i.e.~at the level of equivariant parameters we have $\lambda_1 + \lambda_2+\lambda_3=\lambda_4 = 0$. We conclude that (\ref{equ sum}) becomes the $T$-equivariant Donaldson-Thomas invariants 
of $n$ points on $D$ which, by \cite[II, Thm.~2]{MNOP}, are equal to 
\begin{equation}\sum_{Z\in\Hilb^n(D)^{T}}\frac{e_T(\Ext^{2}_D(I_{Z,D},I_{Z,D}))}{e_T(\Ext^{1}_D(I_{Z,D},I_{Z,D}))} \,q^n=
M(-q)^{\mathlarger{\int}_Dc_3^{T}(TD\otimes K_D)}. \nonumber \end{equation}
By the definition of equivariant push-forward (Def. \ref{equi push}), we have
\begin{eqnarray*}\int_Xc^T_3(X)\cdot c_1^T(L)&:=& \sum_{x\in X^{T}}\frac{\iota_x^*(c^T_3(X)\cdot c_1^T(L))}{c^T_4(T_xX)}  \\
&=& \sum_{x\in X^{T}}\frac{c^T_3(T_xX)\cdot c_1^T(L|_x)}{c^T_4(T_xX)} \\
&=& \sum_{x\in D^{T}}\frac{c^T_3(T_xX)\cdot c_1^T(L|_x)}{c^T_4(T_xX)}, 
\end{eqnarray*}
where $\iota_x: \{x\}\times_T ET\to X\times_T ET$ is the natural inclusion and the last equality follows from Lemma \ref{vanishing} below.
Similarly, we have 
\begin{eqnarray*}\int_Dc^T_3(TD\otimes K_D)&:=& \sum_{x\in D^{T}}\frac{\iota_x^*(c^T_3(TD\otimes K_D))}{c^T_3(T_xD)}  \\
&=& \sum_{x\in D^{T}}\frac{c^T_3(T_xD\otimes K_D|_x)}{c^T_3(T_xD)}. 
\end{eqnarray*}
From the $T$-equivariant short exact sequence
\begin{equation}0\to TD\to TX|_D\to K_D\to 0, \nonumber \end{equation}
we obtain
\begin{align*}
&c^T_3(T_xX)=c^T_3(T_xD)+ c_2^T(T_xD)\cdot c_1^T(K_D|_x), \quad c^T_4(T_xX)=c^T_3(T_xD)\cdot c_1^T(K_D|_x), \\ 
&c_3^T(T_xD \otimes K_D|_x) = c^T_3(T_xD)+ c_2^T(T_xD)\cdot c_1^T(K_D|_x) +c_1^T(T_xD) \cdot c_1^T(K_D|_x)^2 + c_1^T(K_D|_x)^3.
\end{align*}
Since $K_D|_x = \wedge^3 T^*_xD$, we have
$$
c_1^T(T_xD) \cdot c_1^T(K_D|_x)^2 + c_1^T(K_D|_x)^3 = (c_1^T(T_xD) + c_1^T(K_D|_x)) \cdot c_1^T(K_D|_x)^2 = 0
$$
and therefore $\int_Xc^T_3(X)\cdot c_1^T(L)=\int_Dc^T_3(TD\otimes K_D)$ for $L=\oO_X(D)$.
\end{proof}
In order to prove (\ref{vani}), let $X =\mathbb{C}^4$ with coordinates $x_1,x_2,x_3,x_4$ such that the action of $t\in(\mathbb{C}^*)^4$ satisfies 
\begin{equation}t \cdot x_i = t_i x_i, \quad \textrm{for all } i=1,2,3,4, \nonumber \end{equation}
and the $(\mathbb{C}^*)^4$-equivariant line bundle $L$ is given by 
\begin{equation}D:= \{x_4=0 \} \subseteq \mathbb{C}^4 \textrm{ }\textrm{and}\textrm{ } L:=\oO(D).  \nonumber \end{equation}
\begin{lem}\label{vanishing}
We have a $(\mathbb{C}^*)^4$-equivariant isomorphism $L^{[n]} \cong \oO^{[n]} \otimes t_4^{-1}$. Moreover, for any $Z \in \Hilb^n(\mathbb{C}^4)^T$ such that $Z$ does not lie scheme theoretically in $D$, we have 
\begin{equation}e_T(L^{[n]}|_{Z}) = 0.  \nonumber \end{equation}
\end{lem}
\begin{proof}
Consider the ideal sheaf $\oO(-D) \subseteq \oO$. This corresponds to the inclusion
\begin{equation}(x_4) \subseteq \mathbb{C}[x_1,x_2,x_3,x_4]  \nonumber \end{equation}
and therefore $\oO(-D) \cong \oO \otimes t_{4}$ and $L \cong \oO \otimes t_{4}^{-1}$. The fibres of $L^{[n]}$ are given by
\begin{equation}L^{[n]}|_{Z} \cong H^0(L|_Z) \cong H^0(\oO_Z) \otimes t_{4}^{-1},  \nonumber \end{equation}
where all isomorphisms are $(\mathbb{C}^*)^4$-equivariant isomorphisms. Hence, we have a $(\mathbb{C}^*)^4$-equivariant isomorphism
\begin{equation}L^{[n]} \cong \oO^{[n]} \otimes t_4^{-1}.  \nonumber \end{equation}

Now suppose $Z \in \Hilb^n(\mathbb{C}^4)$ is a $T$-fixed (and therefore $(\mathbb{C}^*)^4$-fixed) element. Then $Z$ corresponds to a solid partitions $\pi = \{\pi_{ijk}\}_{i,j,k \geqslant 1}$. Suppose $Z \not\subseteq D$, i.e.~$Z$ is not scheme theoretically contained in $D$, then 
$(x_4) \not\subseteq I_Z$. Therefore, $\pi_{111} > 1$ and the class of $Z$ in the $(\mathbb{C}^*)^4$-equivariant $K$-group $K_{(\mathbb{C}^*)^4}(\bullet)$ contains the term $t_4$. Hence
\begin{equation*}
e_{(\mathbb{C}^*)^4}(L^{[n]}|_{Z}) = e_{(\mathbb{C}^*)^4}(Z \otimes t_{4}^{-1}) = e_{(\mathbb{C}^*)^4}(1 + \mathrm{other \ terms}) = e_{(\mathbb{C}^*)^4}(1) \, e_{(\mathbb{C}^*)^4}(\mathrm{other \ terms}) = 0. 
\end{equation*}
This equality holds for $T$-equivariant Euler classes as well, which corresponds to setting $\lambda_4=-(\lambda_1+\lambda_2+\lambda_3)$.
\end{proof}

\subsection{Vertex formalism}

In order to prove Conjecture \ref{conj for toric}, it is in fact enough to prove it for affine space $\mathbb{C}^4$.
In this section, we develop the necessary vertex formalism from which this follows. We follow the original arguments developed in the 3-dimensional case by MNOP \cite{MNOP} very closely.

Let $X$ be a smooth quasi-projective toric Calabi-Yau 4-fold and let $\{U_\alpha\}$ be the cover by maximal $(\mathbb{C}^*)^4$-invariant affine open subsets. Let $Z \subseteq X$ be a $T$-invariant zero-dimensional subscheme (hence also $(\mathbb{C}^*)^4$-invariant by Lemma \ref{Tfixlocus}). For each $\alpha$, the restriction $Z_\alpha := Z|_{U_\alpha}$ corresponds to a solid partitions $\pi^{(\alpha)}$, as described previously, and we write
$$
I_\alpha := I_{Z_{\pi^{(\alpha)}}}.
$$ 
By footnote \ref{equivtrace}, we have $T$-equivariant trace maps and we can take the trace-free part
$$
-\dR\mathrm{Hom}_X(I_Z,I_Z)_0 \in K_T(\bullet).
$$
Denote the global section functor by $\Gamma(-)$. 
The local-to-global spectral sequence and calculation of sheaf cohomology with respect to the \v{C}ech cover $\{U_\alpha\}$ yields
$$
-\dR\mathrm{Hom}_X(I_Z,I_Z)_0 = \sum_{\alpha,i} (-1)^i \Big(\Gamma(U_\alpha, \oO_{U_\alpha}) -  \Gamma(U_\alpha, \mathcal{E}{\it{xt}}^i(I_\alpha,I_\alpha)) \Big).
$$
Here we use $H^{>0}(U_\alpha,-) = 0$, because $U_\alpha$ is affine. We also use that intersections $U_{\alpha} \cap U_{\beta} \cap \cdots$, 
with $\alpha \neq \beta$, do not contribute because $Z$ is zero-dimensional and therefore
$$
I_Z|_{U_{\alpha} \cap U_{\beta} \cap \cdots} = \oO_{U_{\alpha} \cap U_{\beta} \cap \cdots}.
$$
This reduced the calculation to 
$$
-\dR\mathrm{Hom}_{U_\alpha}(I_\alpha,I_\alpha)_0 = \sum_i (-1)^i \Big( \Gamma(U_\alpha, \oO_{U_\alpha}) -  \Gamma(U_\alpha, \mathcal{E}{\it{xt}}^i(I_\alpha,I_\alpha)) \Big).
$$
On $U_\alpha \cong \mathbb{C}^4$, we use coordinates $x_1,x_2,x_3, x_4$ such that the $(\mathbb{C}^*)^4$-action is given by
$$
t \cdot x_i = t_i x_i, \quad \textrm{for all } i=1,2,3,4. 
$$
Let $U:=U_\alpha$, $Z:=Z_\alpha$, $I:=I_\alpha$, $\pi := \pi^{(\alpha)}$, and $R:=\Gamma(\oO_{U_\alpha}) \cong \mathbb{C}[x_1,x_2,x_3,x_4]$. Consider class $[I]$ in the equivariant $K$-group $K_{(\mathbb{C}^*)^4}(U)$. By identifying $[R]$ with $1$, we obtain a ring isomorphism
$$
K_{(\mathbb{C}^*)^4}(U) \cong \mathbb{Z}[t_1^{\pm},t_2^{\pm},t_3^{\pm},t_4^{\pm}].
$$
The Laurent polynomial $\mathsf{P}(I)$ corresponding to $[I]$ under this isomorphism is called the Poincar\'e polynomial of $I$.
For any $w=(w_1,w_2,w_3,w_4) \in \mathbb{Z}^4$, we use multi-index notation
$$
t^w:= t_1^{w_1} t_2^{w_2} t_3^{w_3} t_4^{w_4}.
$$
Then $[R \otimes t^w] \in K_{(\mathbb{C}^*)^4}(U)$ corresponds to $t^w \in \mathbb{Z}[t_1^{\pm},t_2^{\pm},t_3^{\pm},t_4^{\pm}]$.
 
Define an involution $\overline{(\cdot)}$ on $K_{(\mathbb{C}^*)^4}(U)$ by $\mathbb{Z}$-linear extension of
$$
\overline{t^w} := t^{-w}.
$$ 
By definition, the trace map
$$
\tr : K_{(\mathbb{C}^*)^4}(U) \rightarrow \mathbb{Z}(\!(t_1,t_2,t_3,t_4)\!)
$$
corresponds to $(\mathbb{C}^*)^4$-equivariant restriction to the fixed point of $U$. 

Take a $(\mathbb{C}^*)^4$-equivariant graded free resolution
$$
0 \rightarrow F_s \rightarrow \cdots \rightarrow F_0 \rightarrow I \rightarrow 0,
$$
as in \cite{MNOP}, where 
$$
F_i = \bigoplus_j R \otimes t^{d_{ij}},
$$
for certain $d_{ij} \in \mathbb{Z}^4$. Then
\begin{equation} \label{Poin}
\mathsf{P}(I) = \sum_{i,j} (-1)^i t^{d_{ij}}.
\end{equation}
The $(\mathbb{C}^*)^4$-character of $\oO_Z$ is given by \eqref{Zpi} and can be expressed in terms of the Poincar\'e polynomial of $I$ as follows
\begin{equation} \label{Z}
Z = \sum_{i,j,k \geqslant1} \sum_{l=1}^{\pi_{ijk}} t_1^{i-1} t_2^{j-1} t_3^{k-1} t_4^{l-1}  = \tr(\oO_U - I) = \frac{1-\mathsf{P}(I)}{(1-t_1)(1-t_2)(1-t_3)(1-t_4)}.
\end{equation}
We deduce
\begin{align*}
\dR\mathrm{Hom}_{U}(I,I) &= \sum_{i,j,k,l} (-1)^{i+k} \mathrm{Hom}(R \otimes t^{d_{ij}},R \otimes t^{d_{kl}}) \\
&= \sum_{i,j,k,l} (-1)^{i+k} R \otimes t^{d_{kl} - d_{ij}} \\
&=\mathsf{P}(I)\overline{\mathsf{P}(I)}\\
\tr_{\dR\mathrm{Hom}_{U}(I,I)}&= \frac{\mathsf{P}(I)\overline{\mathsf{P}(I)} }{(1-t_1)(1-t_2)(1-t_3)(1-t_4)},
\end{align*}
where we used \eqref{Poin} for the third equality. Eliminating $\mathsf{P}(I)$ by using \eqref{Z}, the trace of
$-\dR\mathrm{Hom}_{U_\alpha}(I_\alpha,I_\alpha)_0$ is then given by
\begin{equation} \label{defV}
\mathsf{V}_\alpha := Z_\alpha + \frac{\overline{Z}_{\alpha}}{t_1t_2t_3t_4} - \frac{Z_\alpha \overline{Z}_\alpha(1-t_1)(1-t_2)(1-t_3)(1-t_4)}{t_1t_2t_3t_4},
\end{equation}
where we re-introduced the index $\alpha$. Summing up, we have proved the following lemma:
\begin{lem} \label{Vlem}
Let $Z \subseteq X$ be a $T$-fixed zero-dimensional subscheme. Then
\begin{equation*} 
\tr_{-\dR\mathrm{Hom}_X(I_Z,I_Z)_0} = \sum_{\alpha} \tr_{-\dR\mathrm{Hom}_{U_\alpha}(I_{Z_\alpha},I_{Z_\alpha})_0}= \sum_\alpha \mathsf{V}_\alpha,
\end{equation*}
where the equivariant vertex $\mathsf{V}_\alpha$ is defined by \eqref{defV}.
\end{lem}

For a fixed $\alpha$, after specialization $t_1t_2t_3t_4=1$, we have 
\begin{align*}
\mathsf{V}_\alpha &= \mathrm{Ext}^1_{U_\alpha}(I_{Z_\alpha},I_{Z_\alpha}) + \mathrm{Ext}^3_{U_\alpha}(I_{Z_\alpha},I_{Z_\alpha}) 
-\mathrm{Ext}^2_{U_\alpha}(I_{Z_\alpha},I_{Z_\alpha}) \\
&= \mathrm{Ext}^1_{U_\alpha}(I_{Z_\alpha},I_{Z_\alpha}) + \mathrm{Ext}^1_{U_\alpha}(I_{Z_\alpha},I_{Z_\alpha})^*-
\mathrm{Ext}^2_{U_\alpha}(I_{Z_\alpha},I_{Z_\alpha}),  
\end{align*}
where each $\mathrm{Ext}^{i}_{U_\alpha}(I_{Z_\alpha},I_{Z_\alpha})$, with $i \neq 0$, is a finite-dimensional $T$-representation by Lemma \ref{lem on toric identify deform-obs} and $\mathrm{Ext}^2_{U_\alpha}(I_{Z_\alpha},I_{Z_\alpha})$ is self-dual. Consequently
\begin{equation}e_T(-\mathsf{V}_\alpha)=(-1)^{\dim_{\mathbb{C}}\mathrm{Ext}^1_{U_\alpha}(I_{Z_\alpha},I_{Z_\alpha})}\cdot
\frac{e_T(\mathrm{Ext}^2_{U_\alpha}(I_{Z_\alpha},I_{Z_\alpha}))}
{e_T(\mathrm{Ext}^1_{U_\alpha}(I_{Z_\alpha},I_{Z_\alpha}))^2}.  \nonumber \end{equation}
Since the Serre duality pairing on $\mathrm{Ext}^2_{U_\alpha}(I_{Z_\alpha},I_{Z_\alpha})$ is $T$-invariant, 
there exists a half Euler class $e_T\big(\mathrm{Ext}^2_{U_\alpha}(I_{Z_\alpha},I_{Z_\alpha}),Q\big)$ as in (\ref{half euler class equivariant}).
By its property (\ref{half Euler}), we know
\begin{equation}e_T\big(\mathrm{Ext}^2_{U_\alpha}(I_{Z_\alpha},I_{Z_\alpha}),Q\big)^2=(-1)^{\frac{1}{2}\dim_{\mathbb{C}}
\mathrm{Ext}^2_{U_\alpha}(I_{Z_\alpha},I_{Z_\alpha})}\cdot e_T(\mathrm{Ext}^2_{U_\alpha}(I_{Z_\alpha},I_{Z_\alpha})). 
\nonumber  \end{equation}
Denoting the length of the zero-dimensional subscheme $Z_\alpha$ by $n_\alpha$ and using $\chi(\oO_{U_\alpha}) - \chi(I_\alpha,I_\alpha) = 2n_\alpha$, we obtain 
\begin{equation}\label{equi square equality}e_T(-\mathsf{V}_\alpha)=(-1)^{n_\alpha}\cdot
\Bigg(\frac{e_T(\mathrm{Ext}^2_{U_\alpha}\big(I_{Z_\alpha},I_{Z_\alpha}),Q\big)}
{e_T(\mathrm{Ext}^1_{U_\alpha}(I_{Z_\alpha},I_{Z_\alpha}))}\Bigg)^2. \end{equation}
\begin{defi} \label{wpi}
Let $\pi$ be a solid partition of size $|\pi|$ and let $\mathsf{V}_\pi$ be the expression defined by \eqref{defV}, where $Z$ is the $T$-invariant 
zero-dimensional subscheme determined by (\ref{Z}).
We define
$$
\mathsf{w}_\pi:=\pm\sqrt{(-1)^{|\pi|}\cdot e_T(-\mathsf{V}_\pi)}
\in \mathbb{Q}(\lambda_1,\lambda_2,\lambda_3,\lambda_4) / (\lambda_1+\lambda_2+\lambda_3+\lambda_4),
$$
i.e.~the square root of $(-1)^{|\pi|}$ times (\ref{equi square equality}). We only define $\mathsf{w}_{\pi}$ up to a sign $\pm$. 
\end{defi}

From Lemma \ref{Vlem} and Definition \ref{wpi}, we conclude:
\begin{prop}\label{equi euler class vertex}
Let $Z \subseteq X$ be a $T$-fixed zero-dimensional subscheme. Suppose the restriction $Z|_{U_\alpha} \subseteq U_\alpha$ corresponds to a solid partition $\pi^{(\alpha)}$. Then
$$
\frac{e_T\big(\Ext^{2}(I_Z,I_Z),Q\big)}{e_T\big(\Ext^{1}(I_Z,I_Z)\big)}  = \pm \prod_{\alpha} \mathsf{w}_{\pi^{(\alpha)}}.
$$
\end{prop}
${}$ 

\subsection*{Insertions}

Let $L$ be a $(\mathbb{C}^*)^4$-equivariant line bundle on $X$. For each $\alpha$, there exists a character $d^{(\alpha)} = (d^{(\alpha)}_1,d^{(\alpha)}_2,d^{(\alpha)}_3,d^{(\alpha)}_4) \in \mathbb{Z}^4$ such that 
$$
L|_{U_\alpha} = \oO_{U_{\alpha}} \otimes t^{d^{(\alpha)}}.
$$
As above, write $U:=U_\alpha$, $d:=d^{(\alpha)}$, and suppose we have the standard torus action $t \cdot x_i = t_i x_i$ for all $i=1,2,3,4$. Let $Z\subseteq U$ be a 0-dimensional $T$-fixed subscheme corresponding to a solid partition $\pi$. Then we define
$$
L_\pi(d_1,d_2,d_3,d_4) := e_T\big( H^0(U,\oO_{Z}\otimes L|_{U})\big) \in \mathbb{Q}(\lambda_1,\lambda_2,\lambda_3,\lambda_4) / (\lambda_1+\lambda_2+\lambda_3+\lambda_4),
$$
where 
$$
H^0(U,\oO_{Z}\otimes L|_{U}) = \sum_{i,j,k \geqslant 1} \sum_{l=1}^{\pi_{ijk}} t_1^{d_1+i-1} t_2^{d_2+j-1} t_3^{d_3+k-1} t_4^{d_4+l-1}.
$$
Then for any $Z \subseteq X$ we have
$$
e_T(L^{[n]})|_Z = \prod_{\alpha}  L_{\pi^{(\alpha)}}(d^{(\alpha)}_1,d^{(\alpha)}_2,d^{(\alpha)}_3,d^{(\alpha)}_4). 
$$
\begin{exam}
Let $Z_\pi = 1+t_1+t_4$. The corresponding solid partition $\pi$ satisfies  
\begin{equation}\pi_{111}=2, \quad \pi_{211}=1, \quad \pi_{ijk}=0,\,\textrm{ }\mathrm{otherwise}. \nonumber \end{equation}
Hence  
$I_{Z_\pi}=\langle x^2_1,x_1x_4, x^2_4, x_2, x_3 \rangle$.
After specialization $t_1t_2t_3t_4=1$, we get  
{\scriptsize{
\begin{align*}
\mathsf{V}_\pi=&\Big(t^{3}_1t^2_2t^2_3-t^3_1t^2_2t_3-t^3_1t_2t^2_3+t^3_1t_2t_3-t_1t^2_2t^2_3+t_1t^2_2t_3+t_1t_2t^2_3\\
&+2t_1t_2t_3-2t_1t_2+2t_1+t_1t^{-1}_3+t_1t^{-1}_2-2t_1t_3-t_1t^{-1}_2t^{-1}_3+t_2+t_3-t_2t_3 \Big)+ \\
&\Big(t^{-3}_1t^{-2}_2t^{-2}_3-t^{-3}_1t^{-2}_2t^{-1}_3-t^{-3}_1t^{-1}_2t^{-2}_3+t^{-3}_1t^{-1}_2t^{-1}_3-t^{-1}_1t^{-2}_2t^{-2}_3
+t^{-1}_1t^{-2}_2t^{-1}_3+t^{-1}_1t^{-1}_2t^{-2}_3 \\
&+2t^{-1}_1t^{-1}_2t^{-1}_3-2t^{-1}_1t^{-1}_2+2t_1^{-1}+t^{-1}_1t_3+t^{-1}_1t_2-2t^{-1}_1t^{-1}_3-t^{-1}_1t_2t_3+t_2^{-1}+t_3^{-1}-t^{-1}_2t^{-1}_3\Big),
\end{align*} }}
where all terms come in Serre dual pairs. One readily calculates
{\scriptsize{
\begin{align*}
&\mathsf{w}_\pi = \pm \frac{(\lambda_1+\lambda_2)^2(\lambda_1+\lambda_3)^2(\lambda_2+\lambda_3)(\lambda_1-\lambda_2-\lambda_3)(\lambda_1+2\lambda_2+2\lambda_3)(3\lambda_1+2\lambda_2+\lambda_3)(3\lambda_1+\lambda_2+2\lambda_3)}{\lambda_1^2 \lambda_2 \lambda_3(\lambda_1-\lambda_2)(\lambda_1-\lambda_3)(\lambda_1+\lambda_2+\lambda_3)^2(\lambda_1+2\lambda_2+\lambda_3)(\lambda_1+\lambda_2+2\lambda_3)(3\lambda_1+\lambda_2+\lambda_3)(3\lambda_1+2\lambda_2+2\lambda_3)}, \\
& \\ 
& L_\pi(d_1,d_2,d_3,d_4) = \big((d_1-d_4)\lambda_1+(d_2-d_4)\lambda_2+(d_3-d_4)\lambda_3\big)\big((d_1-d_4+1)\lambda_1+(d_2-d_4)\lambda_2+(d_3-d_4)\lambda_3\big) \\
&\qquad\qquad\qquad\qquad\quad \cdot\big((d_1-d_4-1)\lambda_1+(d_2-d_4-1)\lambda_2+(d_3-d_4-1)\lambda_3\big),
\end{align*}
}}
where we used $\lambda_4 = -\lambda_1-\lambda_2-\lambda_3$.
\end{exam}
The following conjecture is a \emph{combinatorial version} of Conjecture \ref{conj for toric} when $X=\mathbb{C}^4$.
\begin{conj} \label{affineconj}
There exists a way of choosing the signs for the equivariant weights $\mathsf{w}_{\pi}$ in Definition \ref{wpi} such that the following identity holds in $\frac{\mathbb{Q}(\lambda_1,\lambda_2,\lambda_3,\lambda_4)}{(\lambda_1+\lambda_2+\lambda_3+\lambda_4)}(d_1,d_2,d_3,d_4)[\![q]\!]$
\begin{equation} 
\sum_{\pi} L_\pi(d_1,d_2,d_3,d_4) \, \mathsf{w}_\pi \,  q^{|\pi|}=M(-q)^{\frac{(d_1 \lambda_1+d_2 \lambda_2+d_3 \lambda_3+d_4 \lambda_4)(-\lambda_1\lambda_2\lambda_3-\lambda_1\lambda_2\lambda_4-\lambda_1\lambda_3\lambda_4-\lambda_2\lambda_3\lambda_4)}{\lambda_1\lambda_2\lambda_3\lambda_4}},
\nonumber \end{equation}
where the sum is over all solid partitions and $M(q)$ denotes the MacMahon function.
\end{conj}
Combining Conjecture \ref{affineconj} with the vertex formalism, we can deduce Conjecture \ref{conj for toric}.
\begin{prop}\label{compare equi conj}
Conjecture \ref{affineconj} is equivalent to Conjecture \ref{conj for toric}.
\end{prop}
\begin{proof}
Conjecture \ref{affineconj} is a special case of Conjecture \ref{conj for toric} when $X=\mathbb{C}^4$. Conversely, assuming
Conjecture \ref{affineconj} is true, we want to prove Conjecture \ref{conj for toric}. 

Let $X$ be a smooth quasi-projective toric Calabi-Yau 4-fold with $(\mathbb{C}^*)^4$-equivariant line bundle $L$. Let $\{U_\alpha\}_{\alpha=1, \ldots e}$ be the cover by maximal open affine $(\mathbb{C}^*)^4$-invariant subsets. Suppose $(\mathbb{C}^*)^4$ acts on the coordinates of $U_\alpha \cong \mathrm{Spec} \, \mathbb{C}[x_1^{(\alpha)}, x_2^{(\alpha)}, x_3^{(\alpha)}, x_4^{(\alpha)}]$ by
$$
t \cdot x_i^{(\alpha)} = \chi_i^{(\alpha)}(t) \, x_i^{(\alpha)}, \quad \textrm{for all } i=1,2,3,4,
$$
for certain characters $\chi_i^{(\alpha)} : (\mathbb{C}^*)^4 \rightarrow \mathbb{C}^{*}$. If $\chi_i^{(\alpha)}(t) = t_i$ is the standard torus action, then 
\begin{align*}
\frac{c_1^T(L|_{p_\alpha}) \, c_3^T(TU_{\alpha}|_{p_\alpha})}{c_4^T(TU_{\alpha}|_{p_\alpha})} = \frac{(d_1 \lambda_1+d_2 \lambda_2+d_3 \lambda_3+d_4 \lambda_4)(-\lambda_1\lambda_2\lambda_3-\lambda_1\lambda_2\lambda_4-\lambda_1\lambda_3\lambda_4-\lambda_2\lambda_3\lambda_4)}{\lambda_1\lambda_2\lambda_3\lambda_4},
\end{align*}
where $p_\alpha=(0,0,0,0) \in U_\alpha$ is the unique $(\mathbb{C}^*)^4$-fixed point. For other characters, the RHS gets adapted accordingly. We deduce
\begin{align*}
& \quad \, \sum_{n=0}^{\infty}\DT_4(X,T,L,n\,;o(\mathcal{L}))\,q^n \\
&= \sum_{n=0}^{\infty}\, q^n \sum_{Z\in\Hilb^n(X)^T}\frac{e_T(\Ext^{2}_X(I_{Z},I_{Z}),Q)\cdot e_T(L^{[n]}|_Z)}{e_T(\Ext^{1}_X(I_{Z},I_{Z}))} \\
&= \sum_{n=0}^{\infty}\, q^n \sum_{Z\in\Hilb^n(X)^{(\mathbb{C}^*)^4}}\frac{e_T(\Ext^{2}_X(I_{Z},I_{Z}),Q)\cdot e_T(L^{[n]}|_Z)}{e_T(\Ext^{1}_X(I_{Z},I_{Z}))} \\
&= \sum_{n_1=0}^{\infty} \sum_{Z_1 \in\Hilb^{n_1}(U_1)^{(\mathbb{C}^*)^4}} \cdots \sum_{n_e=0}^{\infty} \sum_{Z_e \in\Hilb^{n_e}(U_e)^{(\mathbb{C}^*)^4}} \prod_{\alpha=1}^{e} q^{n_{\alpha}} \frac{e_T(\Ext^{2}_{U_{\alpha}}(I_{Z_{\alpha}},I_{Z_{\alpha}}),Q)\cdot e_T(L^{[n_{\alpha}]}|_{Z_{\alpha}})}{e_T(\Ext^{1}_{U_{\alpha}}(I_{Z_{\alpha}},I_{Z_{\alpha}}))} \\
&= \prod_\alpha  \sum_{n_\alpha=0}^{\infty} q^{n_\alpha} \sum_{Z_\alpha \in\Hilb^{n_\alpha}(U_\alpha)^{(\mathbb{C}^*)^4}}\frac{e_T(\Ext^{2}_{U_\alpha}(I_{Z_\alpha},I_{Z_\alpha}),Q)\cdot e_T(L^{[n]}|_{Z_\alpha})}{e_T(\Ext^{1}_{U_\alpha}(I_{Z_\alpha},I_{Z_\alpha}))} \\
&=\prod_\alpha \sum_{\mathrm{solid\, partitions}\, \pi^{(\alpha)}} L_{\pi^{(\alpha)}}(d_1^{(\alpha)},d_2^{(\alpha)},d_3^{(\alpha)},d_4^{(\alpha)}) \,  \mathsf{w}_{\pi^{(\alpha)}} \, q^{|\pi^{(\alpha)}|} \\
&=\prod_\alpha M(-q)^{\frac{c_1^T(L|_{p_\alpha}) c_3^T(TU_{\alpha}|_{p_\alpha})}{c_4^T(TU_{\alpha}|_{p_\alpha})}} =M(-q)^{\sum_{\alpha} \frac{c_1^T(L|_{p_\alpha}) c_3^T(TU_{\alpha}|_{p_\alpha})}{c_4^T(TU_{\alpha}|_{p_\alpha})}} =M(-q)^{\int_X c_1^T(L) c^T_3(T_X)}.
\end{align*}
Here for each $Z \in \Hilb^{n}(X)^{T}$, the signs of $e_T(\Ext^{2}_X(I_{Z},I_{Z}),Q)$ are induced from the choice of signs of $\{e_T(\Ext^{2}_{U_{\alpha}}(I_{Z_{\alpha}},I_{Z_{\alpha}}),Q)\}_\alpha$ when taking the square root of the following equation
\begin{align*}
&(-1)^{\frac{\chi(I_Z,I_Z)_0}{2}} \frac{e_T(\Ext^{2}_X(I_{Z},I_{Z}))}{e_T(\Ext^{1}_X(I_{Z},I_{Z}))\,e_T(\Ext^{3}_X(I_{Z},I_{Z}))} \\
&= \prod_\alpha  (-1)^{\frac{\chi(I_{Z_\alpha},I_{Z_\alpha})_0}{2}} \frac{e_T(\Ext^{2}_{U_\alpha}(I_{Z_\alpha},I_{Z_\alpha}))}{e_T(\Ext^{1}_{U_\alpha}(I_{Z_\alpha},I_{Z_\alpha}))\,e_T(\Ext^{3}_{U_\alpha}(I_{Z_\alpha},I_{Z_\alpha}))}.
\end{align*}
In turn, the signs of $\{e_T(\Ext^{2}_{U_{\alpha}}(I_{Z_{\alpha}},I_{Z_{\alpha}}),Q)\}_\alpha$ are determined by the signs of $\{\mathsf{w}_{\pi^{(\alpha)}}\}_\alpha$ provided by Conjecture \ref{affineconj} (via Definition \ref{wpi} and Proposition \ref{equi euler class vertex}).
\end{proof}
We implemented the calculation of $\mathsf{w}_{\pi}$ in Definition \ref{wpi} into a Maple program. Using this in the context of Conjecture \ref{affineconj} leads us to conjecture the following:
\begin{conj}\label{affineconj unique}
There exists a \emph{unique} way of choosing the signs for the equivariant weights $\mathsf{w}_{\pi}$ such that
Conjecture \ref{affineconj} holds.
\end{conj}
Using our Maple program, we checked the following:
\begin{thm}\label{verify affine conj}
Conjectures \ref{affineconj} and \ref{affineconj unique} are true modulo $q^7$. 
\end{thm}

\begin{rmk}  \label{existencesgns}
A priori there are many possible choices of orientation, i.e.~signs for $\mathsf{w}_\pi$, in Conjecture \ref{affineconj}. E.g.~there are 140 solid partitions of size 6, so in this case there are $2^{140} \approx 10^{42}$ choices! However, we have a (conjectural) very quick way of finding orientations which work. In fact, Conjecture \ref{specconj} of the next section asserts that the specialization $L_\pi(0,0,0,-d) \, \mathsf{w}_\pi$ with $\lambda_1+\lambda_2+\lambda_3=0$ is well-defined (and we check this in many cases). This specialization is conjecturally equal to $(-1)^{|\pi|} \prod_{l=1}^{\pi_{111}} (d-(l-1))$ times a non-zero rational number. By choosing the sign of $\mathsf{w}_\pi$ in such a way that this rational number is \emph{positive}, we end up with \emph{existence} of a collection of signs for which Conjecture \ref{affineconj} holds in the cases that we checked, i.e.~modulo $q^7$. For order $q^6$, the calculation can be efficiently organized by comparing the coefficients of each monomial $d_1^{i_1}d_2^{i_2}d_3^{i_3}d_4^{i_4}$ separately.
\end{rmk}

\begin{rmk} 
For orders $q^{\leqslant 3}$ we check brute force that the choices of orientation, i.e.~signs for $\mathsf{w}_\pi$, in Conjecture \ref{affineconj} are \emph{unique}. For orders $q^4, q^5, q^6$, we first specialize to $d_1=d_2=d_3=0$, $d_4=-d$, $\lambda_1+\lambda_2+\lambda_3=0$ (after observing that this specialization is well-defined) in which case LHS and RHS of Conjecture \ref{affineconj} become polynomials of degree $\delta=4,5,6$ respectively. We then compare the coefficients of the terms of the polynomials starting with the leading term: $d^\delta, d^{\delta-1}, \cdots,d$. It turns out that each comparison uniquely determines some of the signs. E.g.~for $q^6$, comparing the coefficients of $d^6$ fixes 1 sign, comparing the coefficients of $d^5$ fixes 3 further signs, comparing the coefficients of $d^4$ fixes 9 further signs, comparing the coefficients of $d^3$ fixes 25 further signs, comparing the coefficients of $d^2$ fixes 54 further signs, and comparing the coefficients of $d$ fixes the last 48 signs.
\end{rmk}

\section{Application to counting solid partitions}

\subsection{Weighted count of solid partitions} 

In this section, we study Conjecture \ref{affineconj} for a special choice of insertions 
\begin{equation}(d_1,d_2,d_3,d_4)=(0,0,0,-d), \quad d\geqslant1. \nonumber \end{equation}
This has applications to enumerating solid partitions. 

For a solid partition $\pi = \{\pi_{ijk}\}_{i,j,k \geqslant 1}$, we refer to $\pi_{111}$ as its \emph{height}. 
By experimental study of many examples (i.e.~Proposition \ref{verify conj on poly dep}), 
we pose the following conjecture:
\begin{conj} \label{specconj}
Let $\pi$ be a solid partition and let $\mathsf{w}_\pi$ be defined using the unique sign in Conjecture \ref{affineconj unique}. Then the following properties hold:
\begin{itemize}
\item[(a)] $L_\pi(0,0,0,-d) \, \mathsf{w}_\pi \in \frac{\mathbb{Q}(\lambda_1,\lambda_2,\lambda_3,\lambda_4,d)}{(\lambda_1+\lambda_2+\lambda_3+\lambda_4)}$ has no pole at $\lambda_4 = -(\lambda_1+\lambda_2+\lambda_3)$.
\item[(b)] The specialization $L_\pi(0,0,0,-d) \, \mathsf{w}_\pi \Big|_{\lambda_1+\lambda_2+\lambda_3=0}$ is independent of $\lambda_1,\lambda_2,\lambda_3$.
\item[(c)] More precisely, there exists a rational number $\omega_\pi \in \mathbb{Q}_{>0}$ (independent of $d$) such that
\begin{equation}  \label{specconjeq}
L_\pi(0,0,0,-d) \, \mathsf{w}_\pi \Big|_{\lambda_1+\lambda_2+\lambda_3=0} = (-1)^{|\pi|} \, \omega_\pi \, \prod_{l=1}^{\pi_{111}} (d-(l-1)).
\end{equation}
In particular, for $d \in \mathbb Z_{>0}$, the LHS vanishes when $\pi_{111} > d$ and otherwise has the same sign as $(-1)^{\pi}$.
\end{itemize}
\end{conj}
Geometrically, this specialization corresponds to taking $X = \mathbb{C}^4$ and $D = \{x_4^d=0\} \subseteq \mathbb{C}^4$. Then $L = \oO(D) \cong \oO \otimes t_4^{-d}$. As we have seen in Proposition \ref{section s}, the canonical section of $L^{[n]}$ on $\Hilb^n(\mathbb{C}^4)$ cuts out the sublocus of zero-dimensional subschemes $Z$ contained in $D$. At the level of $T$-fixed (and therefore $(\mathbb{C}^*)^4$-fixed) points, this means we are considering solid partitions $\pi$ of height $\pi_{111}\leqslant d$. This is the geometric motivation for the specialization of Conjecture \ref{specconj}. \\

We give the following evidence for Conjecture \ref{specconj}:
\begin{prop} \label{verify conj on poly dep}
\hfill
\begin{itemize}
\item Conjecture \ref{specconj} is true for any solid partition $\pi$ of size $|\pi|\leqslant 6$.
\item Properties (a), (b), and the absolute value of equation \eqref{specconjeq} hold for $d=1$ and any solid partition $\pi$ satisfying $\pi_{111}=1$ (in this case $|\omega_\pi| = 1$).      
\item Properties (a), (b), and the absolute value of equation \eqref{specconjeq} hold for various individual solid partitions of size $\leqslant15$ listed in Appendix A.
\end{itemize}
\end{prop}

\begin{proof}
The second statement follows from Theorem \ref{thm for equiv sm div} and \cite[I, Sect.~4]{MNOP}.
For the other cases, we use our Maple program, which calculates $\mathsf{w}_{\pi}$ for any given solid partition $\pi$. For the first statement, we use the unique choice of signs that we found when verifying Conjecture \ref{affineconj unique} (Theorem \ref{verify affine conj}). 
\end{proof}

Combining Conjectures \ref{specconj} and \ref{affineconj}, we obtain a generating function counting \textit{weighted} solid partitions: 

\begin{thm}\label{solid partition counting}
Assume Conjectures \ref{affineconj} and \ref{specconj} are true. Then
\begin{equation}\label{solid part count}
\sum_{\pi} \omega_\pi\,t^{\,\pi_{111}}\, q^{|\pi|} = e^{t(M(q)-1)},
\end{equation}
where the sum is over all solid partitions, $t$ is a formal variable, and $M(q)$ denotes the MacMahon function. In particular, when $t=1$, we have 
\begin{equation}
\sum_{\pi} \omega_\pi\, q^{|\pi|} = e^{M(q)-1}. \nonumber 
\end{equation}
\end{thm}
\begin{proof}
Consider Conjecture \ref{affineconj} for $d_1=d_2=d_3=0$, $d_4=-d$, and the specialization
\begin{equation}\lambda_1+\lambda_2+\lambda_3=0. \nonumber \end{equation} 
Then the power of $M(-q)$ in Conjecture \ref{affineconj} becomes $d$. According to Conjecture \ref{specconj}, this specialization is well-defined and we get
$$
\sum_{\pi} \omega_\pi  \cdot \Bigg( \prod_{l=1}^{\pi_{111}} (d-(l-1)) \Bigg) q^{|\pi|}= M(q)^d,
$$
for any $d\geqslant1$. Then it is easy to see that
\begin{equation}1+\sum_{\pi_{111}=1} \omega_\pi\,q^{|\pi|} = M(q), \nonumber \end{equation}
\begin{equation}1+2\sum_{\pi_{111}=1} \omega_\pi\,q^{|\pi|}+2!\sum_{\pi_{111}=2} \omega_\pi\,q^{|\pi|} = M(q)^2, \nonumber \end{equation}
\begin{equation}1+3\sum_{\pi_{111}=1} \omega_\pi\,q^{|\pi|}+3\times 2\sum_{\pi_{111}=2} \omega_\pi\,q^{|\pi|}+3!\sum_{\pi_{111}=3} \omega_\pi\,q^{|\pi|} = M(q)^3, 
\nonumber \end{equation}
\begin{equation}\ldots\nonumber \end{equation}
\begin{equation}1+\sum_{i=1}^{k}\frac{k!}{(k-i)!}\sum_{\pi_{111}=i} \omega_\pi\,q^{|\pi|} = M(q)^k, \quad k\geqslant1. \nonumber \end{equation}
Rearranging gives
\begin{equation}t\sum_{\pi_{111}=1} \omega_\pi\,q^{|\pi|} =t(M(q)-1), \nonumber \end{equation}
\begin{equation}t^2\sum_{\pi_{111}=2} \omega_\pi\,q^{|\pi|} =\frac{t^2}{2}\big(M(q)^2-2M(q)+1\big), \nonumber \end{equation}
\begin{equation}t^3\sum_{\pi_{111}=3} \omega_\pi\,q^{|\pi|} =\frac{t^3}{3!}\big(M(q)^3-3M(q)^2+3M(q)-1\big), \nonumber \end{equation}
\begin{equation}\ldots\nonumber \end{equation}
\begin{equation}t^k\sum_{\pi_{111}=k} \omega_\pi\,q^{|\pi|} =\frac{t^k}{k!}\big(M(q)-1\big)^k, \quad k\geqslant1, \nonumber \end{equation}
whose summation gives the equality we want.
\end{proof}
\begin{rmk}
Counting solid partitions is a very difficult question. In fact, MacMahon initially proposed an \emph{incorrect} formula  for its generating function \cite{ABMM}
$$
\sum_{\pi} q^{|\pi|} \stackrel{?}{=} \prod_{n=1}^{\infty} \frac{1}{(1-q^n)^{\frac{1}{2} n(n+1)}}.
$$  
Exact enumeration using computers also does not go very far. 
As Stanley wrote in his PhD thesis \cite{Stanley}\,\footnote{This quote is taken from slides of a talk by S.~Govindarajan, 
Aspects of Mathematics, IMSc, Chennai (2014).} \\

\noindent \textit{``The case $r=2$ has a well-developed theory --- here 2-dimensional partitions are known as plane partitions. (...)
For $r \geqslant 3$, almost nothing is known and (...)
casts only a faint glimmer of light on a vast darkness.''} \\

\noindent We find that a specialization of the weights $L(d_1,d_2,d_3,d_4)_\pi \, \mathsf{w}_{\pi}$, coming naturally from $\DT_4$ theory, gives a weighted count of solid partitions with a nice closed formula \eqref{solid part count}. Of course, one can always find $\omega_\pi$ 
such that \eqref{solid part count} holds (e.g.~simply by expanding the RHS of (\ref{solid part count}) and giving all solid partitions of the same size and height an equal weight). 

Below we will find an explicit (conjectural) formula of $\omega_\pi$ for any solid partition $\pi$ (see Conjecture \ref{compare wpi} and Proposition \ref{evidence compare wpi}). In terms of this explicit formula, it actually becomes rather elementary to prove the counterpart of Theorem \ref{solid partition counting} (i.e.~Proposition \ref{combinatorial wpi thm}). 
Nevertheless, we find it interesting that such weights $\omega_\pi$ naturally arise from $\DT_4$ theory, even though they may have limited combinatorial interest. 
\end{rmk}

\subsection{Combinatorial approach to $\omega_\pi$} 

In this section, we assign an explicit weight $\omega_\pi^c$ to any solid partition (Definition \ref{combinatoric wpi}). Firstly, we unconditionally prove the 
analogue of Theorem \ref{solid partition counting} with $\omega_\pi$ replaced by $\omega_\pi^c$ (Proposition \ref{combinatorial wpi thm}). Secondly, an obvious generalization of Proposition \ref{combinatorial wpi thm} turns out to hold for partitions of any dimension $d$ (Remark \ref{rmk on d dim combinatoric weight}). Thirdly, we conjecture $\omega_\pi = \omega_\pi^c$, for any solid partition $\pi$, and we verify this in many examples (Conjecture \ref{compare wpi} and Proposition \ref{evidence compare wpi}).

\begin{defi}
Let $\xi= \{\xi_{ij}\}_{i,j \geqslant1}$ be a plane partition, i.e.~a sequence of non-negative integers satisfying
\begin{align*}
&\xi_{ij} \geqslant \xi_{i+1,j}, \quad \xi_{ij} \geqslant \xi_{i,j+1}, \quad \forall \, i,j \geqslant 1, \\ 
&|\xi| := \sum_{i,j} \xi_{ij} < \infty. 
\end{align*}
We define the \emph{binary representation} of $\xi$ to be the sequence of integers 
$\{\xi(i,j,k)\}_{i,j,k\geqslant 1}$ given by
$$
\xi(i,j,k) := \left\{\begin{array}{cc} 1 & \textrm{ if\ } k\leqslant \xi_{ij} \\ 0 & \textrm{otherwise.} \end{array}\right.
$$
\end{defi}
\begin{exam}
Suppose $\xi$ is given by $\xi_{11}=2$, $\xi_{21}=1$, $\xi_{12}=1$. Then $\xi(1,1,1)=\xi(1,1,2)=\xi(2,1,1) = \xi(1,2,1) = 1$ and $\xi(i,j,k) = 0$ for all other $i,j,k \geqslant 1$.
\end{exam}
\begin{defi}\label{combinatoric wpi}
Let $\pi = \{\pi_{ijk}\}_{i,j,k \geqslant 1}$ be a (non-empty) solid partition and consider all possible sequences of integers $\{m_\xi\}_\xi$, where the index $\xi$ runs over all (non-empty) plane partitions and $m_\xi \in \mathbb{Z}_{\geqslant 0}$. Define the following collection
\begin{equation} \label{eqnpixi}
\mathcal{C}_\pi := \Bigg\{ \{m_\xi\}_\xi \ \Bigg| \ \pi_{ijk} = \sum_\xi m_\xi \cdot \xi(i,j,k) \ \textrm{for all} \ i,j,k \Bigg\}.
\end{equation}
We define
\begin{equation} \label{summxi}
\omega^c_\pi := \sum_{\{m_\xi\}_\xi \in \mathcal{C}_\pi} \prod_\xi \frac{1}{(m_\xi)!}.
\end{equation}
For the empty solid partition $\pi = \varnothing$ we define $\omega_\pi^c :=1$.
\end{defi}
\begin{rmk}
For each $\{m_\xi\}_\xi\in\mathcal{C}_\pi$, we have 
\begin{equation}|\pi|=\sum_{\xi}m_{\xi}\cdot |\xi|. \nonumber \end{equation}
Hence, $m_{\xi}=0$ if $|\xi|$ is large. Therefore, the collection $\mathcal{C}_{\pi}$ is a finite set and, for each $\{m_\xi\}_\xi\in\mathcal{C}_\pi$, 
there are only finitely many nonzero $m_{\xi}$.
\end{rmk}

\begin{exam}
Suppose $\pi = \{\pi_{ijk}\}_{i,j,k \geqslant 1}$ satisfies $\pi_{ijk} = 0$ unless $i=j= 1$. Then
$$\omega^c_\pi = \prod_{k=1}^{\infty} \frac{1}{(\pi_{11k} - \pi_{11,k+1})!}.$$ 
This is due to the fact that the only plane partitions $\xi=\{\xi_{ijk}\}_{i,j,k \geqslant 1}$ contributing to the defining equation in \eqref{eqnpixi} satisfy $\xi(i,j,k) = 0$ unless $i=j=1$. Define $\xi^{(n)}$ to be the plane partition with binary representation satisfying $\xi(1,1,k) = 1$ for all $1 \leqslant k \leqslant n$ and $\xi(i,j,k) = 0$ otherwise. Then $\mathcal{C}_{\pi}$ only consists of one element $\{m_\xi\}_\xi$:
$$
m_\xi = \left\{\begin{array}{cc} \pi_{11k} - \pi_{11,k+1} & \textrm{if } \xi = \xi^{(k)} \\ 0 & \textrm{otherwise.} \end{array} \right.
$$  
\end{exam}

\begin{exam}
Consider the solid partition $\pi$ of Example \ref{maxsquared}, i.e.~$\pi_{111}=2$, $\pi_{211} = \pi_{121} = \pi_{112} = 1$, and $\pi_{ijk} = 0$ otherwise. Then 
$
\omega_\pi^c = 4.
$
Indeed $\mathcal{C}_{\pi}$ contains the following four sequences, each contributing 1 to the sum in \eqref{summxi}:
\begin{itemize}
\item Consider the plane partitions $\xi^{(1)}$ and $\xi^{(2)}$ defined by the following binary representations: $\xi^{(1)}(1,1,1) = \xi^{(1)}(2,1,1) = \xi^{(1)}(1,2,1) = \xi^{(1)}(1,1,2) = 1$ and $\xi^{(1)}(i,j,k)=0$ otherwise; $\xi^{(2)}(1,1,1) = 1$ and $\xi^{(2)}(i,j,k)=0$ otherwise. Define $\{m_\xi\}_\xi$ by $m_\xi = 1$ if $\xi = \xi^{(1)}$ or $\xi^{(2)}$ and $m_\xi = 0$ otherwise.
\item Consider the plane partitions $\xi^{(1)}$ and $\xi^{(2)}$ defined by the following binary representations: $\xi^{(1)}(1,1,1) = \xi^{(1)}(2,1,1) = \xi^{(1)}(1,2,1) = 1$ and $\xi^{(1)}(i,j,k)=0$ otherwise; $\xi^{(2)}(1,1,1) = \xi^{(2)}(1,1,2) = 1$ and $\xi^{(2)}(i,j,k)=0$ otherwise. Define $\{m_\xi\}_\xi$ by $m_\xi = 1$ if $\xi = \xi^{(1)}$ or $\xi^{(2)}$ and $m_\xi = 0$ otherwise.
\item Consider the plane partitions $\xi^{(1)}$ and $\xi^{(2)}$ defined by the following binary representations: $\xi^{(1)}(1,1,1) = \xi^{(1)}(2,1,1) = \xi^{(1)}(1,1,2) = 1$ and $\xi^{(1)}(i,j,k)=0$ otherwise; $\xi^{(2)}(1,1,1) = \xi^{(2)}(1,2,1) = 1$ and $\xi^{(2)}(i,j,k)=0$ otherwise. Define $\{m_\xi\}_\xi$ by $m_\xi = 1$ if $\xi = \xi^{(1)}$ or $\xi^{(2)}$ and $m_\xi = 0$ otherwise.
\item Consider the plane partitions $\xi^{(1)}$ and $\xi^{(2)}$ defined by the following binary representations: $\xi^{(1)}(1,1,1) = \xi^{(1)}(1,2,1) = \xi^{(1)}(1,1,2) = 1$ and $\xi^{(1)}(i,j,k)=0$ otherwise; $\xi^{(2)}(1,1,1) = \xi^{(2)}(2,1,1) = 1$ and $\xi^{(2)}(i,j,k)=0$ otherwise. Define $\{m_\xi\}_\xi$ by $m_\xi = 1$ if $\xi = \xi^{(1)}$ or $\xi^{(2)}$ and $m_\xi = 0$ otherwise.
\end{itemize} 
\end{exam}

The combinatorial weights $\omega_\pi^c$ lead to the following generating series:
\begin{prop} \label{combinatorial wpi thm}
The following identity holds
$$
\sum_{\pi} \omega^c_\pi \,t^{\pi_{111}} \, q^{|\pi|} = e^{t(M(q)-1)},
$$
where the sum is over all solid partitions, $t$ is a formal variable, and $M(q)$ denotes the MacMahon function. In particular, when $t=1$, we have
\begin{equation}
\sum_{\pi} \omega^c_\pi\, q^{|\pi|} = e^{M(q)-1}. \nonumber 
\end{equation}
\end{prop}

\begin{proof}
The RHS can be rewritten as
\begin{equation} \label{prodprod}
\Bigg(\prod_{\xi \vdash 1} e^{tq} \Bigg) \Bigg( \prod_{\xi \vdash 2} e^{tq^2} \Bigg) \Bigg(\prod_{\xi \vdash 3} e^{tq^3} \Bigg) \cdots
\end{equation}
where $\prod_{\xi \vdash n}$ denotes the finite product over all plane partitions $\xi$ of size $n$. 

Choose a sequence of multiplicities $\{m_\xi\in \mathbb{Z}_{\geqslant 0}\}_\xi$ with only finitely many $m_\xi > 0$. This choice gives rise to a solid partition $\pi$ defined as follows
$$
\pi_{ijk} := \sum_\xi m_\xi \cdot \xi(i,j,k), \quad \textrm{for all } i,j,k \geqslant 1, 
$$
which we call the \emph{solid partition associated to $\{m_\xi\}_\xi$}. Conversely, for a fixed solid partition $\pi$, we can consider the collection of all sequences $\{m_\xi\in \mathbb{Z}_{\geqslant 0}\}_\xi$ with only finitely many $m_\xi > 0$ whose associated solid partition is $\pi$. This collection is precisely $\mathcal{C}_\pi$. 

Each term arising from multiplying out the infinite product \eqref{prodprod} corresponds to a sequence $\{m_\xi\in \mathbb{Z}_{\geqslant 0}\}_\xi$ with only finitely many $m_\xi > 0$.  Such a term contributes
\begin{equation} \label{oneterm}
\prod_{\xi} \frac{t^{m_\xi}}{(m_\xi) !}\, q^{m_\xi |\xi|}.
\end{equation}
Now collect all terms of the form \eqref{oneterm} such that $\{m_\xi\}_\xi$ has associated solid partition $\pi$. This gives
\begin{align*}
\sum_{\{m_\xi\}_\xi \in \mathcal{C}_\pi} \prod_{\xi} \frac{t^{m_\xi}}{(m_\xi) !}q^{m_\xi |\xi|} &= \Bigg(\sum_{\{m_\xi\}_\xi \in \mathcal{C}_\pi} \prod_{\xi} \frac{1}{(m_\xi) !} \Bigg)t^{\pi_{111}}q^{|\pi|}= \omega^c_\pi\,t^{\pi_{111}} q^{|\pi|},
\end{align*}
where we use $\sum_{\xi} m_\xi=\pi_{111}$ in the first equality.
Summing over all distinct solid partitions gives the formula of the proposition.
\end{proof}

\begin{rmk}\label{rmk on d dim combinatoric weight}
We may also start with $d$-partitions\footnote{E.g.~1-partitions are partitions, 2-partitions are plane partitions, 3-partitions are solid partitions.}
$\pi$ for any $d \geqslant 1$ and define $\omega^c_\pi$ completely analogously using $(d-1)$-partitions $\xi$ and their binary representations. 
The same proof yields$$
\log \sum_{d\textrm{-partitions}\, \pi} \omega^c_\pi \,t^{\pi_{111}} \, q^{|\pi|} = t\sum_{(d-1)\textrm{-partitions}\, \pi,\, |\pi|\geqslant1} q^{|\pi|},
$$
where we use the convention that there exists a single zero-dimensional partition of each size.  
\end{rmk}

We end this section with the observation that a specialization of $\DT_4$ theory precisely seems to recover the combinatorics that we just described (and this is how we found the weights $\omega_\pi^c$ in the first place).
\begin{conj}\label{compare wpi}
For any solid partition $\pi$, we have $\omega_\pi = \omega_\pi^c$, where $\omega_\pi$ is defined using $\DT_4$ theory in Conjecture \ref{specconj} and $\omega^c_\pi$ is the explicit combinatorial weight of Definition \ref{combinatoric wpi}.
\end{conj}

Using our Maple program, which calculates $\mathsf{w}_{\pi}$ for a given $\pi$, we verified the following:
\begin{prop}  \label{evidence compare wpi}
\hfill
\begin{itemize}
\item Conjecture \ref{compare wpi} is true for any solid partition $\pi$ of size $|\pi|\leqslant 6$.
\item $|\omega_\pi| = \omega_\pi^c$ for any solid partition $\pi$ satisfying $\pi_{111}=1$.     
\item $|\omega_\pi| = \omega_\pi^c$ for the explicit list of solid partitions of size $\leqslant15$ given in Appendix A.
\end{itemize}
\end{prop}

\appendix

\section{Explicit calculations of $|\omega_\pi|$}

Using our Maple program, which calculates $\mathsf{w}_\pi$ for a given solid partition $\pi$, we checked that
\begin{align}
\begin{split} \label{appendixeqn}
L_\pi(0,0,0,-d) \, \mathsf{w}_\pi \Big|_{\lambda_1+\lambda_2+\lambda_3=0} &= (-1)^{|\pi|} \, \omega_\pi \, \prod_{l=1}^{\pi_{111}} (d-(l-1)), \\
\omega_\pi  &= \omega_\pi^c, 
\end{split}
\end{align}
hold for all solid partitions $\pi$ with $|\pi| \leqslant 6$. Here the signs of $\mathsf{w}_\pi$ are the ones induced from Conjecture \ref{affineconj unique}. 
We also checked that the \emph{absolute value} of equations \eqref{appendixeqn} hold for:
\begin{itemize}
\item \textbf{(Height 1 and $d=1$)} Let $\pi$ be a solid partition with $\pi_{111}=1$. Then $|\omega_\pi| = \omega_\pi^c = 1$.
\item \textbf{(1-Partitions of size $\leqslant 10$)} All solid partitions $\pi = \{\pi_{ijk}\}_{i,j,k \geqslant 1}$ with $\pi_{ijk} = 0$ unless $i=j=1$ and $|\pi| \leqslant10$. Then 
$$|\omega_\pi| = \omega^c_\pi = \prod_{k=1}^{\infty} \frac{1}{(\pi_{11k} - \pi_{11,k+1})!}.$$ 
\item \textbf{(Size 7)} Consider the solid partition $\pi$ corresponding to
$$
Z_\pi = 1+t_1+t_2+t_1t_2+t_3+t_4+t_4^2. 
$$
Then $|\omega_\pi| = \omega_\pi^c = \frac{3}{2}$.
\item \textbf{(Size 8)} Consider the solid partition $\pi$ corresponding to
$$
Z_\pi = 1+t_1+t_2+t_1t_2+t_3+t_4+t_1 t_4+t_4^2. 
$$
Then $|\omega_\pi| = \omega_\pi^c = 3$.
\item \textbf{(Size 9)} Consider the solid partition $\pi$ corresponding to
$$
Z_\pi = 1+t_1+t_1^2+t_2+t_1t_2+t_3+t_4+t_1 t_4+t_4^2. 
$$
Then $|\omega_\pi| = \omega_\pi^c = 6$.
\item \textbf{(Size 10)} Consider the solid partition $\pi$ corresponding to
$$
Z_\pi = 1+t_1+t_1^2+t_2+t_1t_2+t_2t_3+t_3+t_4+t_1 t_4+t_4^2. 
$$
Then $|\omega_\pi| = \omega_\pi^c = 2$.
\item \textbf{(Size 11)} Consider the solid partition $\pi$ corresponding to
$$
Z_\pi = 1+t_1+t_1^2+t_2+t_1t_2+t_2t_3+t_3+t_4+t_1 t_4+t_2t_4+t_4^2. 
$$
Then $|\omega_\pi| = \omega_\pi^c = 8$.
\item \textbf{(Size 12)} Consider the solid partition $\pi$ corresponding to
$$
Z_\pi = 1+t_1+t_1^2+t_2+t_1t_2+t_2t_3+t_3+t_4+t_1 t_4+t_2t_4+t_4^2+ t_4^3. 
$$
Then $|\omega_\pi| = \omega_\pi^c = 6$.
\item \textbf{(Size 13)} Consider the solid partition $\pi$ corresponding to
$$
Z_\pi = 1+t_1+t_1^2+t_2+t_1t_2+t_2t_3+t_3+t_4+t_1 t_4+t_2t_4+t_4^2+ t_4^3+ t_4^4. 
$$
Then $|\omega_\pi| = \omega_\pi^c = \frac{8}{3}$.
\item \textbf{(Size 14)} Consider the solid partition $\pi$ corresponding to
$$
Z_\pi = 1+t_1+t_1^2+t_2+t_1t_2+t_2t_3+t_3+t_4+t_1 t_4+t_2t_4+t_4^2+ t_4^3+ t_4^4 + t_4^5. 
$$
Then $|\omega_\pi| = \omega_\pi^c = \frac{5}{6}$.
\item \textbf{(Size 15)} Consider the solid partition $\pi$ corresponding to
$$
Z_\pi = 1+t_1+t_1^2+t_2 + t_2^2+t_1t_2+t_2t_3+t_3+t_4+t_1 t_4+t_2t_4+t_4^2+ t_4^3+ t_4^4 + t_4^5. 
$$
Then $|\omega_\pi| = \omega_\pi^c = \frac{5}{3}$.
\end{itemize}

\section{Nekrasov's conjecture}

The first author heard the following related conjecture (written below in terms of equivariant $\DT_4$ theory) from Professor Nikita Nekrasov during a visit to the Simons Center for Geometry and Physics in October 2016. For a recent much more general K-theoretical version, see \cite{Nekrasov}.

Let $X=\mathbb{C}^4$ and let $T=\{t\in (\mathbb{C}^*)^4\,|\, t_1t_2t_3t_4=1\}$ be the Calabi-Yau torus. Denote the equivariant parameters of $(\mathbb{C}^*)^4$ by $\lambda_i$ ($i=1,2,3,4$). We define
\begin{equation}[\Hilb^n(\mathbb{C}^4)]^{\mathrm{vir}}_{T,o(\mathcal{L})}:=\sum_{Z\in\Hilb^n(\mathbb{C}^4)^T} \frac{e_T\big(\Ext^{2}(I_Z,I_Z),Q\big)}{e_T\big(\Ext^{1}(I_Z,I_Z)\big)}. \nonumber \end{equation} 
As in Definition \ref{def of equ virtual class}, this depends on a choice of orientation $o(\mathcal{L})$ as in Definition \ref{def of equ virtual class} which is used to define the half Euler classes. Consider the generating function
\begin{equation}Z_{\mathbb{C}^4}:=\sum_{n=0}^{\infty}\Big(\int_{[\Hilb^n(\mathbb{C}^4)]^{\mathrm{vir}}_{T,o(\mathcal{L})}}1\Big)\cdot q^n
\in \frac{\mathbb{Q}(\lambda_1,\lambda_2,\lambda_3,\lambda_4)}{(\lambda_1+\lambda_2+\lambda_3+\lambda_4)}[\![q]\!]. 
\nonumber \end{equation}
\begin{conj}\label{nek conj}
There exist choices of orientation such that 
\begin{equation}Z_{\mathbb{C}^4}=e^{\frac{(\lambda_1+\lambda_2)(\lambda_1+\lambda_3)(\lambda_2+\lambda_3)}{\lambda_1\lambda_2\lambda_3(\lambda_1+\lambda_2+\lambda_3)}q}. \nonumber \end{equation}
\end{conj}
Using the signs discussed in Remark \ref{existencesgns}, we checked the following with our Maple program:
\begin{prop}
Conjecture \ref{nek conj} is true modulo $q^{7}$.
\end{prop}
In fact, the signs of Nekrasov's conjecture seem to be unique as well:
\begin{prop}
Modulo $q^5$, there are unique choices of signs for which Conjecture \ref{nek conj} holds.
\end{prop}

\end{document}